\documentclass{amsart}
\usepackage{amssymb}
\usepackage{amsmath,enumerate}
\usepackage{graphicx}
\usepackage{color}
\newtheorem{Def}{Definition}[section]
\theoremstyle{remark}
\newtheorem{Rem}[Def]{Remark}
\newtheorem{Ex}[Def]{Example}

\theoremstyle{plain}
\newtheorem{Th}[Def]{Theorem}
\newtheorem{Prop}[Def]{Proposition}
\newtheorem{Lem}[Def]{Lemma}
\newtheorem{Cor}[Def]{Corollary}
\newtheorem{Fact}[Def]{Fact}

\newcommand{\Q}{\mathbb{Q}}

\newcommand{\Z}{\mathbb{Z}}

\newcommand{\C}{\mathbb{C}}

\newcommand{\CM}{\mathcal{M}}

\newcommand{\al}{\alpha }
\newcommand{\be}{\beta }
\newcommand{\ga}{\gamma }
\newcommand{\de}{\delta }
\newcommand{\Ga}{\Gamma }

\newcommand{\vep}{\varepsilon }

\newcommand{\la}{\lambda }

\newcommand{\pa}{\partial }
\newcommand{\ot}{\otimes }

\newcommand{\zero}{\mathbf{0}}

\newcommand{\tF}{\widetilde{F}}

\newcommand{\tf}{\widetilde{f}}

\newcommand{\ztwo}{ \{ 0,1 \} }
\newcommand{\ztwon}{ \{ 0,1 \}^n }

\newcommand{\bfv}{\boldsymbol{v}}
\newcommand{\sol}{Sol_{\dot{x}}}

\def\tp#1{\mathord{\mathopen{{\vphantom{#1}}^t}#1}} 

\newcommand{\ee}{\mathbf{e}} 
\newcommand{\Mon}{\mathbf{Mon}} \newcommand{\Ref}{\mathbf{Ref}}

\newcommand{\GL}{\mathbf{GL}} 
\newcommand{\BFH}{\mathbf{H}}

\title[Finite irreducible monodromy of $F_C$]{
  Lauricella's $F_C$ with finite irreducible monodromy group
}
\author[Y. Goto]{Yoshiaki Goto}
\address[Goto]{
  General Education,
  Otaru University of Commerce,
  Otaru 047-8501, Japan
}
\email{goto@res.otaru-uc.ac.jp}

\date{\today}

\begin{document}
\begin{abstract}
  This paper presents our study of the conditions under which the monodromy group for 
  Lauricella's hypergeometric function $F_C (a,b,c;x)$ is finite irreducible. 
  We provide these conditions in terms of the parameters $a,b,c$.
  In addition, we discuss the structure of the finite irreducible monodromy group. 
\end{abstract}

\maketitle

\section{Introduction}\label{section-intro}
Lauricella's hypergeometric series $F_C$ of $n$ variables $x_1 ,\ldots ,x_n$ 
with complex parameters $a$, $b$, $c_1, \ldots, c_n$ is defined by 
\begin{align*}
 F_C (a,b,c ;x ) 
 =\sum_{m_1 ,\ldots ,m_n =0} ^{\infty } 
 \frac{(a,m_1 +\cdots +m_n )(b,m_1 +\cdots +m_n )}
 {(c_1 ,m_1 )\cdots (c_n ,m_n ) m_1 ! \cdots m_n !} x_1 ^{m_1} \cdots x_n^{m_n} ,
\end{align*}
where $x=(x_1 ,\ldots ,x_n),\ c=(c_1 ,\ldots ,c_n)$, 
$c_1 ,\ldots ,c_n \not\in \{ 0,-1,-2,\ldots \}$, and $(c_1 ,m_1)=\Gamma (c_1+m_1)/\Gamma (c_1)$. 
This series converges in the domain 
$$
D_C =\left\{ (x_1 ,\ldots ,x_n ) \in \C^n  \ \middle| \ \sum_{k=1}^{n} \sqrt{|x_k|} <1  \right\} .
$$
% In \cite{HT}, it was shown that the hypergeometric system $E_C (a,b,c)$ 
% of differential equations (see (\ref{eq:system-FC})) satisfied by $F_C (a,b,c;x)$ is a holonomic system 
The hypergeometric system $E_C (a,b,c)$ 
of differential equations (see (\ref{eq:system-FC})) satisfied by $F_C (a,b,c;x)$ was shown \cite{HT} to be a holonomic system
of rank $2^n$ with the singular locus 
\begin{align*}
  % \label{sing-locus}
  S= \Big( \prod_{k=1}^n x_k \cdot R(x)=0 \Big) \subset \C^n ,\quad
  % \nonumber
  R(x_1 ,\ldots ,x_n)=\prod_{\vep_1 ,\ldots ,\vep_n =\pm 1} 
  \Big( 1+\sum_{k=1}^n \vep_k \sqrt{x_k} \Big), 
\end{align*}
and that the system $E_C(a,b,c)$ is irreducible 
if and only if the parameters $a,b,c$ satisfy
\begin{align}
  \label{irred-1}
  a-\sum_{k=1}^n i_k c_k ,\quad b-\sum_{k=1}^n i_k c_k \not\in \Z ,\qquad 
  \forall I=(i_1,\dots,i_n)\in \ztwon.
\end{align}
In \cite{GM-FC}, we constructed a fundamental system $\{ \tF_I\}$ of solutions to $E_C (a,b,c)$ 
in a simply connected domain in $D_C -S$ 
under the condition (\ref{irred-1}); 
for details, see Fact \ref{fact:basis}.  

Let $X$ be the complement of the singular locus $S$. 
The fundamental group of $X$ is generated by $n+1$ loops 
$\rho_0 , \rho_1 ,\ldots , \rho_n$ (see Subsection \ref{section-MR}).
In \cite{G-FC-monodromy},
we expressed 
the circuit transformations $\CM_i$ along $\rho_i$ $(i=0,\dots , n)$
by the intersection form on twisted homology groups associated with
the Euler-type integrals of solutions to $E_C (a,b,c)$. 
In \cite{GM-FC}, we also obtained their representation matrices 
$M_i$ $(i=0,\dots , n)$ 
% of these circuit transformations 
with respect to the basis $\{ \tF_I\}$. 
These matrices are of simple forms. 

In this paper, we present our study of 
the monodromy group $\Mon$, which is a subgroup of $\GL_{2^n}(\C)$ 
generated by these representation matrices. 
When $n=2$, Lauricella's $F_C$ is also known as Appell's $F_4$.  
Several studies have been conducted on 
the monodromy group. 
% for Appell's $F_4$.
For example, the finite monodromy group was studied in \cite{Kato} and \cite{Kato-Sekiguchi}, 
and the Zariski closure of $\Mon$, 
which is the Picard--Vessiot group, was studied in \cite{Sasaki}. 

In \cite{GK-FC-Zariski}, we previously investigated the Zariski closure of $\Mon$ 
for general $n$. 
% , by using the reflection subgroup $\Ref$. 
In this study, as a generalization of \cite{Kato}, 
we provide the conditions under which $\Mon$ is finite irreducible. 
As was mentioned in \cite[Proposition 2.14]{GK-FC-Zariski}, 
the (in)finiteness conditions are important 
to classify the Zariski closure of $\Mon$. 
The conditions for the finite irreducible monodromy group 
are given as follows 
(another formulation is given in Theorem \ref{th:fin-irr}). 
\begin{Th}%[Main theorem]
  \label{th:fin-irr-intro}
  We assume $n\geq 3$. 
  The monodromy group $\Mon$ is finite irreducible 
  if and only if the following two conditions hold:
  \begin{enumerate}
  \item[(\ref{condition-A})] 
    for each $k=1,\dots ,n$, 
    the monodromy group for Gauss' hypergeometric differential equation 
    ${}_2 E_1 (a,b,c_k)$ is
    finite irreducible; 
  \item[(\ref{condition-B})] 
    at least $n$ of $c_1 ,\dots ,c_n$, $b-a$, $c_1 +\cdots +c_n-a-b-(n-1)/2$ 
    are equivalent to $1/2$ modulo $\Z$. 
    % are elements in $\Z+1/2$. 
  \end{enumerate}
\end{Th}
We prove this theorem by focusing on the reflection subgroup $\Ref$, which 
is a normal subgroup generated by the reflection $M_0$ (see Section \ref{section:Ref}). 
Certain concepts in our proofs are based on those of \cite{Kato}. 
However, we note that the condition (\ref{condition-B}) in Theorem \ref{th:fin-irr-intro}
is not a direct generalization of \cite{Kato} (see Remark \ref{rem:F4-FC} (ii)). 

The finiteness condition is also known as the algebraicity condition. 
Namely, the monodromy group is finite if and only if the solutions to $E_C (a,b,c)$ 
are algebraic functions over $\C (x_1,\dots ,x_n)$. 
% In \cite[Theorem 2.4.12]{Bod}, 
In \cite{Bod}, 
algebraicity conditions of $E_C(a,b,c)$ were determined by using 
results of \cite{Beukers}, which were obtained by studying the algebraicity conditions for 
$A$-hypergeometric systems. 
Their approaches are quite different from ours. 
The monodromy group was not treated directly in \cite{Beukers} and \cite{Bod}, 
whereas in our study, we investigate it in detail. 
Further, using our results for the reflection subgroup,
we can also determine the structure of the finite irreducible monodromy group. 
% in Section \ref{section-str}.

% The author thanks to the referee for suggesting many improvements in the earlier version of the article.

\section{Preliminaries}
% We assume the conditions 
% for parameters $a,b,c_1,\dots,c_n$ in (\ref{irred-1}) 
% (it is equivalent to (\ref{irred-2}) mentioned below). 
In this section,
we present certain pertinent facts about Lauricella's $F_C$ 
mentioned in previous studies 
(\cite{GM-FC}, \cite{HT}, and \cite{L}). 
% we collect some facts about Lauricella's $F_C$ 
% mentioned in 
% % \cite{G-FC-monodromy}, 
% \cite{GM-FC}, \cite{HT}, and \cite{L}. 
We set
\begin{align*}
  \al =\exp (2\pi \sqrt{-1} a),\quad 
  \be =\exp (2\pi \sqrt{-1} b),\quad 
  \ga_k =\exp (2\pi \sqrt{-1} c_k)\ (k=1,\ldots ,n).
\end{align*}
Under these notations, the condition (\ref{irred-1}) is equivalent to 
\begin{align}
  \label{irred-2}
  \al-\prod_{k=1}^n \ga_k^{i_k} ,\quad \be-\prod_{k=1}^n \ga_k^{i_k} \neq 0 ,\qquad 
  \forall I=(i_1,\dots,i_n) \in\ztwon.
\end{align}
% For example, $\ga_k =1$ or $\al \be -(-1)^{n-1}\prod_{k=1}^n \ga_k =0$ are allowed. 

% Note that though in \cite{G-FC-monodromy} the indices $I$ run the subsets of $\{ 1,\dots ,n \}$, 
% in this paper we use $\ztwon$ as a set of indices. 
% The correspondence is given by 
% $$
% \{ 1,\dots ,n \} \supset \{ i_1 ,\dots ,i_r \} \longleftrightarrow 
% e_{i_1}+\dots+e_{i_r}\in \ztwon ,
% $$
% where $e_k$ is the $k$-th unit vector of size $n$.
% We put $|I|=\sum_{k=1}^n i_k$.

\subsection{System of differential equations}\label{subsection:DE}
For $k=1 ,\dots , n$, 
let $\pa_k$ be the partial differential operator with respect to $x_k$. 
We set $\theta_k =x_k \pa_k$ and $\theta =\sum_{k=1}^n \theta_k$.  
Lauricella's $ F_C (a,b,c;x)$ satisfies the system of differential equations 
\begin{equation}
\label{eq:system-FC}  
\left[ \theta_k (\theta_k+c_k-1)-x_k(\theta +a)(\theta +b)  \right] f(x)=0 
\quad (k=1,\dots , n).
\end{equation}
The system (\ref{eq:system-FC}) is 
known as Lauricella's hypergeometric system $E_C (a,b,c)$ of 
differential equations. 
By \cite{HT},
% the system $E_{C} (a,b,c)$ is a holonomic system of rank $2^n$ 
the left ideal generated by the differential operators (\ref{eq:system-FC}) is a holonomic ideal of rank $2^n$ 
with the singular locus $S$, 
and the system is irreducible
% the system $E_C(a,b,c)$ is irreducible
% that is, the system $E_{C} (a,b,c)$ defines a maximal ideal in the ring of 
% differential operators with rational function coefficients, 
if and only if the parameters $a,b,c_1,\dots,c_n$ satisfy 
(\ref{irred-1}) (equivalently, $\al,\be,\ga_1,\dots,\ga_n$ satisfy 
(\ref{irred-2})).

Set $\dot{x}=\left( \frac{1}{2n^2},\ldots ,\frac{1}{2n^2} \right) \in X$, 
and let $\sol$ be the local solution space to $E_C(a,b,c)$ around $\dot{x}$.
For $I=(i_1,\dots,i_n) \in \ztwon$, we set 
\begin{align*}
% \label{series-sol}
F_I(x)=\frac{\prod_{k=1}^n\Ga((-1)^{i_k}(1-c_k))}
{\Ga(1-a^I)\Ga(1-b^I)}\cdot \prod_{k=1}^n x_k^{i_k(1-c_k)}
\cdot F_C(a^I,b^I,c^I;x),
\end{align*}
where
\begin{align*}
  &a^I=a+\sum_{k=1}^n i_k(1-c_k),\quad 
    b^I=b+\sum_{k=1}^n i_k(1-c_k), \\
  &c^I=(c_1+2i_1(1-c_1),\dots,c_n+2i_n(1-c_n)).
\end{align*}
It is known that 
the functions $\{ F_I\}_{I\in \ztwon}$ form 
the basis of $\sol$ under the conditions (\ref{irred-1}) and $c_1,\dots,c_n\notin \Z$.

\subsection{Monodromy representation and fundamental group}
\label{section-MR}
For $\rho \in \pi_1 (X,\dot{x})$ and $g \in \sol$, 
let $\rho_* g$ be the analytic continuation of $g$ along $\rho$. 
Since $\rho_* g$ is also a solution to $E_C (a,b,c)$,  
the map $\rho_* :\sol \to \sol ;\ g\mapsto \rho_* g$ defines 
a linear automorphism. 
Thus, we obtain the \emph{monodromy representation} 
$$
\CM :\pi_1 (X,\dot{x}) \to GL(\sol);\ 
\rho \mapsto \rho_{*}
$$
of $E_C(a,b,c)$, where $GL(V)$ is 
the general linear group on a vector space $V$. 

Next, we introduce the generators of the fundamental group $\pi_1 (X,\dot{x})$. 
Let $\rho_0, \rho_1 ,\ldots ,\rho_n$ be loops in $X$ such that 
\begin{itemize}
\item $\rho_0$ turns around the hypersurface $(R(x)=0)$ near the point 
  $\left( \frac{1}{n^2},\ldots, \frac{1}{n^2} \right)\in (R(x)=0)$, positively; 
\item $\rho_k \ (k=1,\dots , n)$ turns around the hyperplane $(x_k=0)$, positively. 
\end{itemize}
The precise definitions can be found in \cite{G-FC-monodromy}.
% For the precise definitions, see \cite{G-FC-monodromy}. 
\begin{Fact}[\cite{G-FC-monodromy}, \cite{GK-FC-pi1}, \cite{Terasoma}]\label{pi1}
  The loops $\rho_0 ,\rho_1 ,\ldots ,\rho_n$ generate 
  the fundamental group $\pi_1 (X,\dot{x})$. 
  If $n\geq 2$, they satisfy 
  \begin{align*}
    \rho_i \rho_j =\rho_j \rho_i \quad (i,j=1,\dots , n) ,\quad
    (\rho_0 \rho_k)^2 =(\rho_k \rho_0)^2 \quad (k=1,\dots , n).
  \end{align*}
  In addition, if $n\geq 3$, they also satisfy
  \begin{align*}
    &(\rho_{i_1} \cdots \rho_{i_p})\rho_0 (\rho_{i_1} \cdots \rho_{i_p})^{-1} \cdot 
    (\rho_{j_1} \cdots \rho_{j_q})\rho_0 (\rho_{j_1} \cdots \rho_{j_q})^{-1} \\
    &=
    (\rho_{j_1} \cdots \rho_{j_q})\rho_0 (\rho_{j_1} \cdots \rho_{j_q})^{-1} \cdot 
    (\rho_{i_1} \cdots \rho_{i_p})\rho_0 (\rho_{i_1} \cdots \rho_{i_p})^{-1} ,
  \end{align*}
  for $I=\{ i_1 ,\ldots, i_p\}$, $J=\{ j_1 ,\ldots, j_q\} \subset \{ 1,\ldots ,n\}$ 
  with $p,q \geq 1$, $p+q \leq n-1$ and $I \cap J=\emptyset$. 
  Further, these relations generate all relations in $\pi_1 (X,\dot{x})$. 
\end{Fact}
In \cite{G-FC-monodromy}, $n+1$ linear maps 
$\CM_i =\CM (\rho_i) \ (i=0,\dots , n)$ were investigated in terms of 
twisted homology groups and the intersection form.

\subsection{Representation matrices of monodromy}
As in \cite{GK-FC-Zariski} and \cite{GM-FC}, 
we define the tensor product $A\otimes B$ of matrices 
$A$ and $B=(b_{ij})_{1\leq i\leq r, 1\leq j\leq s}$ as 
\begin{align*}
  A\otimes B=
  \begin{pmatrix}
    A\, b_{11} & A\, b_{12} &\cdots & A\, b_{1s} \\
    A\, b_{21} & A\, b_{22} &\cdots & A\, b_{2s} \\
    \vdots  &\vdots   &\ddots & \vdots  \\
    A\, b_{r1} & A\, b_{r2} &\cdots & A\, b_{rs} \\
  \end{pmatrix}.
\end{align*}
% We remark that this is different from the usual definition. 
We regard $\C^{2^n}$ as $\C^2 \otimes \cdots \otimes \C^2$, and 
take as basis 
\begin{align*}
  \ee_I=\ee_{i_1, \dots, i_n} = \ee_{i_1} \otimes \cdots \otimes \ee_{i_n}, \qquad
  &\ee_0 = \begin{pmatrix} 1 \\ 0 \end{pmatrix}, \ 
  \ee_1 = \begin{pmatrix} 0 \\ 1 \end{pmatrix}, \\ 
  &I=(i_1,\dots,i_n) \in \ztwon. 
\end{align*}
We align this basis in the order $\prec$ of 
indices $I = (i_1, \dots, i_n) \in \{0,1\}^n$, which is given by 
\begin{align*}
  &(i_1 ,\dots ,i_n) \prec (i'_1 ,\dots ,i'_n) \\
  &\Longleftrightarrow  
  \exists r ~ \textrm{ s.t. } ~ i_j =i'_j \ (j=r+1 ,\ldots ,n),\ 
  i_r=0,\ i'_r=1 .
\end{align*}
For example, we align the basis of $\C^{8} =\C^2 \otimes\C^2 \otimes\C^2$ ($n=3$) as follows: 
\begin{align*}
  (\ee_{0,0,0}, \ee_{1,0,0}, \ee_{0,1,0}, \ee_{1,1,0}, 
  \ee_{0,0,1}, \ee_{1,0,1}, \ee_{0,1,1}, \ee_{1,1,1}). 
\end{align*}
% \begin{align*}
%   (0,\dots,0),\ (1,0,\dots,0), \ (0,1,\dots,0),\  (1,1,\dots,0),\ 
%   (0,0,1,\dots,0),\ \dots, (1,\dots,1).
% \end{align*}

\begin{Fact}[{\cite[Theorem 3.3]{GM-FC}}]\label{fact:basis}
  We define $\{ \tF_I\}_{I \in \ztwon}$ to be 
  \begin{align*}
    (\dots, \tF_I(x),\dots) = (\dots, F_I(x),\dots) \cdot \left(
    \begin{pmatrix}
      1- \ga_1& 1 \\
      0 & 1 
    \end{pmatrix}
    \ot \cdots \ot
    \begin{pmatrix}
      1- \ga_n& 1 \\
      0 & 1 
    \end{pmatrix}
    \right).
  \end{align*}
  Then, $\{ \tF_I\}_{I\in \ztwon}$ form 
  the basis of $\sol$ under the condition (\ref{irred-1}) only. 
\end{Fact}
% As Subsection \ref{subsection:DE}, we take the basis $\{ F_I\}_I$ of $\sol$.  
In \cite{GM-FC}, the representation matrices of $\CM_i$'s 
with respect to the basis $\{ \tF_I\}_I$ were obtained; they are simple matrices. 
Let $E_m$ be a unit matrix of size $m$. 
\begin{Fact}[{\cite[Corollary 3.5]{GM-FC}}]
  \label{fact:rep-mat}
  Let $M_i$ 
  be the representation matrix of $\CM_i$ $(i=0, \dots ,n)$
  with respect to the basis $\{ \tF_I\}_I$. 
  For $k=1,\dots, n$, we have 
  \begin{align*}
    M_k =E_2 \ot \cdots \ot E_2 \ot \underset{k\textrm{th}}{G(\ga_k)} \ot E_2 \ot \cdots \ot E_2 ,
    \quad 
    G(\ga_k) =\begin{pmatrix}1 & -\ga_k^{-1} \\ 0 & \ga_k^{-1} \end{pmatrix}.    
  \end{align*}
  The matrix $M_0$ is written as 
  \begin{align*}
    M_0 =E_{2^n} -\tp{(\zero, \dots ,\zero ,\bfv)} ,
    % M_0 =E_{2^n} -N_0 ,\quad 
    % N_0 =\tp{(\zero, \dots ,\zero ,\bfv)} ,
  \end{align*}
  where $\bfv \in \C^{2^n}$ is a column vector of which the $I$-th entry is 
  \begin{align*}
    \left\{ 
      \begin{array}{ll}
        (-1)^n \frac{(\al-1)(\be-1)\prod_{k=1}^n \ga_k}{\al \be}& 
       (I=(0,\dots ,0)) ,\\
        (-1)^{n+|I|} \frac{(\al \be +(-1)^{|I|} \prod_{k=1}^n \ga_k^{i_k})
          \prod_{k=1}^n \ga_k^{1-i_k}}{\al \be}& (I\neq(0,\dots ,0)) .
      \end{array}
    \right.
  \end{align*}
\end{Fact}
% In these expressions, we can see that 
The aforementioned expressions indicate that
these representation matrices depend on the parameters
$\al$, $\be$, $\ga_1,\dots ,\ga_n$. 
In other words, they are determined by the parameters $a$, $b$, $c_1,\dots ,c_n$ modulo $\Z$. 
Thus, we often write $M_i = M_i^{(n)}(\al ,\be ,\ga)$ and  
$\bfv =\bfv^{(n)}(\al ,\be ,\ga) =\tp{ (\ldots ,v^{(n)}_I(\al ,\be ,\ga) ,\ldots)}$. 

\begin{Rem}\label{rem:M0-reflection}
  $\ee_{1,\dots,1} =\ee_1 \otimes \cdots \otimes \ee_1$ is an eigenvector of 
  $M_0=M_0^{(n)}(\al ,\be ,\ga)$, 
  that is,  
  \begin{align*}
    M_0 \ee_{1,\dots,1} = \delta_0^{(n)} (\al ,\be ,\ga) \ee_{1,\dots,1} ,\quad 
    \delta_0^{(n)} (\al ,\be ,\ga)  = (-1)^{n+1} \frac{\gamma_1 \cdots \gamma_n}{\alpha \beta} .
  \end{align*}
  The eigenspace of $M_0$ with eigenvalue $1$ is $2^n -1$ dimensional. 
  The matrix $M_0^{(n)} (\al ,\be ,\ga)$ is a ``reflection'' (see Section \ref{section:Ref}) 
  with the special eigenvalue $\delta_0^{(n)} (\al ,\be ,\ga)$. 
\end{Rem}

\begin{Ex}
  In the case $n=2$, the representation matrices are as follows. 
{\allowdisplaybreaks
\begin{align*}
  M_1&=
  \begin{pmatrix}1 & -\ga_1^{-1} \\ 0 & \ga_1^{-1} \end{pmatrix}
  \ot 
  \begin{pmatrix}1 & 0 \\ 0 & 1 \end{pmatrix}
  =
  \begin{pmatrix}
    1&-\ga_{1}^{-1}&0&0\\ 
    0&\ga_{1}^{-1}&0&0\\  
    0&0&1&-\ga_{1}^{-1}\\ 
    0&0&0&\ga_{1}^{-1}
  \end{pmatrix},\\
  M_2&=
  \begin{pmatrix}1 & 0 \\ 0 & 1 \end{pmatrix}
  \ot 
  \begin{pmatrix}1 & -\ga_2^{-1} \\ 0 & \ga_2^{-1} \end{pmatrix}
  =
  \begin{pmatrix}
    1&0&-{\ga_{2}}^{-1}&0 \\ 
    0&1&0&-\ga_{2}^{-1}\\ 
    0&0&\ga_{2}^{-1}&0 \\
    0&0&0&\ga_{2}^{-1}
  \end{pmatrix}, \\
  M_0&=
  \begin{pmatrix}
    1&0&0&0\\ 
    0&1&0&0\\ 
    0&0&1&0\\ 
    -\frac{(\al-1)(\be-1) \ga_{1}\ga_{2}  }{\al\be}
    &\frac { \left( \al\be-\ga_{1} \right) \ga_{2}}{\al\be}
    &\frac{\left( \al\be-\ga_{2} \right)\ga_{1}}{\al\be}
    &-\frac{\ga_{1}\ga_{2}}{\al \be}
  \end{pmatrix}.
\end{align*}
}
\end{Ex}

\subsection{Monodromy group}
Using the basis $\{ \tF_I\}_I$, we can identify $\sol$ and $\C^{2^n}$. 
Thus, we regard the monodromy representation as a group homomorphism $\CM :\pi_1 (X,\dot{x}) \to \GL_{2^n}(\C)$. 
The \emph{monodromy group} $\Mon=\Mon^{(n)}(\al,\be,\ga)$ is defined by 
\begin{align*}
  \Mon^{(n)}(\al,\be,\ga) =\CM (\pi_1 (X,\dot{x}))
  =\langle M_0 ,M_1 ,\dots ,M_n\rangle .
\end{align*}
Recall that the matrices $M_0 ,M_1 ,\dots ,M_n$ depend on 
$\al$, $\be$, $\ga_1 ,\dots ,\ga_n$.
We restate Theorem \ref{th:fin-irr-intro} in terms of $\al$, $\be$, and $\ga$. 
\begin{Th}
  \label{th:fin-irr}
  We assume $n\geq 3$. 
  The monodromy group $\Mon^{(n)}(\al,\be,\ga)$ is finite irreducible 
  if and only if the following two conditions hold:
  \begin{enumerate}[(A)]
  \item\label{condition-A} 
    each $\Mon^{(1)}(\al,\be,\ga_k)$ ($k=1,\dots ,n$) is 
    finite irreducible; 
  \item\label{condition-B} 
    at least $n$ of $\ga_1 ,\dots ,\ga_n$, $\be\al^{-1}$, $\de_0^{(n)}(\al,\be,\ga)$ 
    are $-1$. 
  \end{enumerate}
\end{Th}
On the other hand, 
for $n=2$ (Appell's $F_4$), the finite irreducible condition was given by Kato \cite{Kato}. 
\begin{Fact}[{\cite[Theorem 1]{Kato}}]
  \label{fact:F4-fin-irr}
  The monodromy group $\Mon^{(2)}(\al,\be,(\ga_1 ,\ga_2 ))$ is finite irreducible 
  if and only if the following two conditions hold:
  \begin{itemize}
  \item[(\ref{condition-A})] 
    $\Mon^{(1)}(\al,\be,\ga_1)$ and $\Mon^{(1)}(\al,\be,\ga_2)$
    are finite irreducible; 
  \item[(\ref{condition-B}')]
    $\de_0^{(2)}(\al,\be,(\ga_1,\ga_2))=-1$, or 
    at least two of $\ga_1$, $\ga_2$, $\be\al^{-1}$ are $-1$. 
  \end{itemize}
\end{Fact}
\begin{Rem}\label{rem:F4-FC}~
  \begin{enumerate}[(i)]
  \item The monodromy group $\Mon^{(1)}(\al,\be,\ga_k)$ is 
    nothing but that for Gauss' hypergeometric function 
    ${}_2 F_1 (a,b,c_k ;x)$. 
    % It is known that 
    % its irreducibility condition is  
    Its irreducibility condition is known to be
    \begin{align*}
      \al-1,\ \al-\ga_k ,\ \be-1, \ \be-\ga_k \neq 0 
    \end{align*}
    (see also (\ref{irred-2})); 
    the finiteness conditions (the so-called ``Schwarz list'') are provided in \cite{Schwarz}.
    Interested readers are referred to \cite{Bod}, in which an accessible list is available.
  \item If $n=2$, then (\ref{condition-A}) and $\de_0^{(2)}(\al,\be,(\ga_1 ,\ga_2))=-1$ imply the finiteness of 
    the monodromy group
    $\Mon^{(2)}(\al,\be,(\ga_1 ,\ga_2))$. 
    However, if $n\geq 3$, (\ref{condition-A}) and $\de_0^{(n)}(\al,\be,\ga)=-1$ are not sufficient 
    for the finiteness. 
    Thus, Theorem \ref{th:fin-irr} is not a direct generalization of 
    Fact \ref{fact:F4-fin-irr}. 
  \end{enumerate}
\end{Rem}

\section{Reflection subgroup}\label{section:Ref}
In this section, 
we assume that the irreducibility condition (\ref{irred-2}) holds. 

As in \cite{BH}, we refer to a matrix $g \in \GL_n(\C)$ as
a \emph{reflection} if $g - E_n$ has rank one. 
We refer to the determinant of a reflection $g$ as the 
special eigenvalue of $g$. 
As mentioned in Remark \ref{rem:M0-reflection}, 
$M_0^{(n)}(\al,\be,\ga)$ is a reflection with 
the special eigenvalue $\de_0^{(n)}(\al,\be,\ga)$. 

Let $\Ref =\Ref^{(n)}(\al,\be,\ga) \subset \Mon$ be the smallest normal subgroup containing $M_0$, 
that is, a subgroup generated by reflections $gM_0g^{-1} \ (g \in \Mon)$. 
The reflection subgroup was introduced in \cite{BH} for the generalized hypergeometric 
function $_{n}F_{n-1}$ and subsequently considered in \cite{Kato} for Appell's $F_4$. 
% to study the finiteness condition. 
Then, we introduced the \emph{reflection subgroup} $\Ref$ for the study of $F_C$
in \cite{GK-FC-Zariski}. 
\begin{Fact}[{\cite[Proposition 2.15]{GK-FC-Zariski}}] \label{fact-ref}
  The monodromy group $\Mon^{(n)}(\al,\be,\ga)$ is finite 
  if and only if $\Ref^{(n)}(\al,\be,\ga)$ is finite.
\end{Fact}
To discuss the finiteness of $\Mon$, it suffices to consider 
that of $\Ref$. 
We use the following two lemmas. 
Although the reducibility was shown in \cite[Lemmas 2.20 and 2.21]{GK-FC-Zariski}, 
we need more precise statements about direct product decompositions. 
\begin{Lem}\label{lem:red-1}
  If at least two of $\ga_1 ,\ldots ,\ga_n$ are $-1$, 
  then the action of $\Ref$ is reducible. 
  For example, if $\ga_{n-1}=\ga_n=-1$, then 
  we have the decomposition 
  \begin{align*}
    \Ref^{(n)}(\al,\be,(\ga_1,\dots ,\ga_{n-2},-1,-1)) \simeq 
    \left(\Ref^{(n-1)}(\al,\be,(\ga_1,\dots ,\ga_{n-2},-1)) \right)^2 .
  \end{align*}
\end{Lem}
\begin{Lem}\label{lem:red-2}
  If at least one of $\ga_1 ,\ldots ,\ga_n$ is $-1$ and 
  $\alpha \beta^{-1}$ is $-1$, 
  then the action of $\Ref$ is reducible. 
  For example, if $\ga_n=\be\al^{-1}=-1$, then 
  we have the decomposition 
  \begin{align*}
    \Ref^{(n)}(\al,\be,(\ga_1,\ga_2,\dots, \ga_{n-1},-1)) \simeq 
    \left(\Ref^{(n-1)}(\al,\be,(\ga_1,\ga_2,\dots ,\ga_{n-1})) \right)^2 .
  \end{align*}
\end{Lem}
To prove these lemmas, we use the same decompositions of $\C^{2^n}$ into 
$\Ref$-invariant subspaces as \cite{GK-FC-Zariski}.  
Recall that for $i,j=1,\dots ,n$, we have 
$M_i M_j =M_j M_i$ by the relation in Fact \ref{pi1}. 

\begin{proof}[Proof of Lemma \ref{lem:red-1}]
  Without loss of generality, 
  we may assume $\ga_{n-1}=\ga_n=-1$. 
  This implies $M_{n-1}^2 =M_n^2 =E_{2^n}$. 
  For each $(i_1 ,\dots ,i_{n-2})\in \ztwo^{n-2}$, we set  
  \begin{align*}
    g_{i_1,\dots ,i_{n-2},0,0} 
    &
      % = \ee_{i_1,\dots ,i_{n-2},0,0} 
      =\ee_{i_1} \otimes \cdots \otimes \ee_{i_{n-2}} \otimes \ee_0 \otimes \ee_0, \\
    g_{i_1,\dots ,i_{n-2},1,0} 
    &
      % = 2\ee_{i_1,\dots ,i_{n-2},1,0}- \ee_{i_1,\dots ,i_{n-2},0,0} 
      =\ee_{i_1} \otimes \cdots \otimes \ee_{i_{n-2}} \otimes (2\ee_1 -\ee_0) \otimes \ee_0 , \\
    g_{i_1,\dots ,i_{n-2},0,1} 
    &
      % = 2\ee_{i_1,\dots ,i_{n-2},0,1}- \ee_{i_1,\dots ,i_{n-2},0,0} 
      =\ee_{i_1} \otimes \cdots \otimes \ee_{i_{n-2}} \otimes \ee_0 \otimes (2\ee_1 -\ee_0) , \\
    g_{i_1,\dots ,i_{n-2},1,1} 
    % &= 4\ee_{i_1,\dots ,i_{n-2},1,1} -2\ee_{i_1,\dots ,i_{n-2},1,0}
    % -2\ee_{i_1,\dots ,i_{n-2},0,1}+ \ee_{i_1,\dots ,i_{n-2},0,0} \\
    &=\ee_{i_1} \otimes \cdots \otimes \ee_{i_{n-2}} \otimes (2\ee_1 -\ee_0) \otimes (2\ee_1 -\ee_0) ,\\
    f_{14;i_1 ,\dots ,i_{n-2}}^{\pm} &=g_{i_1,\dots ,i_{n-2},0,0} \pm g_{i_1,\dots ,i_{n-2},1,1} , \\
    f_{23;i_1 ,\dots ,i_{n-2}}^{\pm} &=g_{i_1,\dots ,i_{n-2},1,0} \pm g_{i_1,\dots ,i_{n-2},0,1} .
  \end{align*}
  We consider a direct sum decomposition of $\C^{2^n}$: 
  \begin{align*}
    \C^{2^n}=W^{+}\oplus W^{-} ;\quad 
    W^{\pm} = \bigoplus_{(i_1 ,\dots ,i_{n-2})} \C f_{14;i_1 ,\dots ,i_{n-2}}^{\pm}
    \oplus \bigoplus_{(i_1 ,\dots ,i_{n-2})} \C f_{23;i_1 ,\dots ,i_{n-2}}^{\pm} .
  \end{align*}
  The dimension of each factor is $2^{n-2} + 2^{n-2}=2^{n-1}$. Note that 
  \begin{align*}
    \ee_{1,\dots ,1}
    =\frac{1}{4} (f_{14;1,\dots ,1}^{+}+f_{23;1,\dots ,1}^{+} ) \in W^{+} .
  \end{align*}
  As was shown in \cite{GK-FC-Zariski}, we obtain the following equalities: 
  \begin{align*}
    &M_k \cdot f_{\ast; i_1,\dots ,i_k,\dots ,i_{n-2}}^{\pm}=
    \begin{cases}
      f_{\ast; i_1,\dots, 0,\dots ,i_{n-2}}^{\pm} & (i_k=0) \\
      -\gamma_k^{-1} f_{\ast; i_1 ,\dots ,0,\dots ,i_{n-2}}^{\pm} 
      +\gamma_k^{-1} f_{\ast; i_1 ,\dots ,1,\dots,i_{n-2}}^{\pm} & (i_k=1),
    \end{cases} \\
    &M_{n-1} \cdot f_{14;i_1 ,\dots ,i_{n-2}}^{\pm} = f_{14;i_1 ,\dots ,i_{n-2}}^{\mp},\quad 
    M_{n-1} \cdot f_{23;i_1 ,\dots ,i_{n-2}}^{\pm} = -f_{23;i_1 ,\dots ,i_{n-2}}^{\mp} , \\ 
    &M_n \cdot f_{\ast;i_1 ,\dots ,i_{n-2}}^{\pm} = f_{\ast;i_1 ,\dots ,i_{n-2}}^{\mp} , \\
    % &M_n \cdot f_{14;i_1 ,\dots ,i_{n-2}}^{\pm} = f_{14;i_1 ,\dots ,i_{n-2}}^{\mp},
    % M_n \cdot f_{23;i_1 ,\dots ,i_{n-2}}^{\pm} = f_{23;i_1 ,\dots ,i_{n-2}}^{\mp} , \\
    &M_0 \cdot f_{\ast;i_1 ,\dots ,i_{n-2}}^{-}=f_{\ast;i_1 ,\dots ,i_{n-2}}^{-}\in W^{-} , \\
    % &M_0 \cdot f_{14;i_1 ,\dots ,i_{n-2}}^{-}=f_{14;i_1 ,\dots ,i_{n-2}}^{-}\in W^{-} , \quad 
    % M_0 \cdot f_{23;i_1 ,\dots ,i_{n-2}}^{-}=f_{23;i_1 ,\dots ,i_{n-2}}^{-}\in W^{-} , \\
    &M_0 \cdot f_{14;i_1 ,\dots ,i_{n-2}}^{+}=f_{14;i_1 ,\dots ,i_{n-2}}^{+}
    -2\lambda_{0;i_1 ,\dots ,i_{n-2}}\ee_{1,\dots ,1} \in W^{+}, \\ 
    &M_0 \cdot f_{23;i_1 ,\dots ,i_{n-2}}^{+}=f_{23;i_1 ,\dots ,i_{n-2}}^{+}
    -2(2\lambda_{1;i_1 ,\dots ,i_{n-2}} -\lambda_{0;i_1 ,\dots ,i_{n-2}})\ee_{1,\dots ,1} \in W^{+} ,
  \end{align*}
  where $\ast =14$ or $23$, $1\leq k \leq n-2$, and 
  \begin{align*}
    &\lambda_{1;i_1 ,\dots ,i_{n-2}} =(-1)^{n+i_1+\dots +i_{n-2}}
    \frac{(\alpha \beta +(-1)^{i_1+\dots +i_{n-2}} 
      \prod_{k=1}^{n-2} \gamma_k^{i_k})\prod_{k=1}^{n-2} \gamma_k^{1-i_k}}{\alpha \beta} ,\\
    &\lambda_{0;i_1 ,\dots ,i_{n-2}} =\left\{ 
      \begin{array}{ll}
        (-1)^n \frac{(\alpha-1)(\beta-1)\prod_{k=1}^{n-2} \gamma_k}{\alpha \beta}
        & ((i_1 ,\dots ,i_{n-2})=(0,\dots ,0)) \\
        \lambda_{1;i_1 ,\dots ,i_{n-2}}
        & ((i_1 ,\dots ,i_{n-2})\neq (0,\dots ,0)) .
      \end{array}
    \right. 
  \end{align*}
  These equalities imply that $W^{\pm}$ are non-trivial $\Ref$-subspaces 
  (see \cite{GK-FC-Zariski}). 
  Because
  \begin{itemize}
  \item $M_0$ and $(M_{n-1}M_n)M_0(M_{n-1}M_n)^{-1}$ act trivially on $W^{-}$,
  \item $M_n M_0 M_n^{-1}$ and $M_{n-1}M_0 M_{n-1}^{-1}$ act trivially on $W^{+}$, 
  \end{itemize}
  we have the direct product decomposition
  \begin{align}
    \label{eq:direct-prod-1}
    &\Ref^{(n)}(\al,\be,(\ga_1,\dots ,\ga_{n-2},-1,-1))\\
    \nonumber
    &=\left\langle \left. 
      \begin{array}{c}
        gM_0 g^{-1},g M_{n-1} M_0 M_{n-1}^{-1} g^{-1}, \\
        g M_n M_0 M_n^{-1} g^{-1},g(M_{n-1}M_n)M_0(M_{n-1}M_n)^{-1} g^{-1}
      \end{array}
    \right| 
    g=M_1^{j_1} \cdots M_{n-2}^{j_{n-2}} ,\ j_k\in \Z
    \right\rangle \\
    \nonumber
    &=\left\langle 
      gM_0 g^{-1},g(M_{n-1}M_n)M_0(M_{n-1}M_n)^{-1} g^{-1}
      \mid 
      g=M_1^{j_1} \cdots M_{n-2}^{j_{n-2}} ,\ j_k\in \Z
      \right\rangle \\
    \nonumber
    & \qquad \times \left\langle 
      g M_{n-1}M_0 M_{n-1}^{-1} g^{-1},
      g M_n M_0 M_n^{-1} g^{-1}
      \mid 
      g=M_1^{j_1} \cdots M_{n-2}^{j_{n-2}} ,\ j_k\in \Z
      \right\rangle \\
    \nonumber
    &=R^{+} \times R^{-} .
  \end{align}
  Here, $R^{+}$ (resp. $R^{-}$) acts trivially on $W^{-}$ (resp. $W^{+}$). 
  We retake the bases of $W^{\pm}$ as 
  \begin{align*}
    \tf_{i_1 ,\dots ,i_{n-2},i_{n-1}}^{\pm}
    =
    \begin{cases}
      f_{14;i_1 ,\dots ,i_{n-2}}^{\pm} & (i_{n-1}=0)\\
      \frac{1}{2}(f_{14;i_1 ,\dots ,i_{n-2}}^{\pm}+f_{23;i_1 ,\dots ,i_{n-2}}^{\pm}) &(i_{n-1}=1), 
    \end{cases}
  \end{align*}
  where $(i_1 ,\dots ,i_{n-2},i_{n-1})\in \{ 0,1\}^{n-1}$.
  Note that $\tf_{1 ,\dots ,1,1}^{+}=2\ee_{1,\dots ,1}$. 
  We consider the representation matrices of the actions by 
  $M_k$ ($1\leq k \leq n-2$) and $M_{n-1}M_n$ on $W^{\pm}$, 
  $M_0$ on $W^{+}$, and $M_n M_0 M_n^{-1}$ on $W^{-}$
  with respect to these bases. 
  \begin{itemize}
  \item For $1\leq k \leq n-2$, the representation matrix of 
    the action by $M_k$ is 
    \begin{align*}
      E_2 \ot \cdots \ot E_2 \ot \underset{k\textrm{th}}{G(\ga_k)} \ot E_2 \ot \cdots \ot 
      \underset{(n-1)\textrm{th}}{E_2} 
      =M_k^{(n-1)}(\al,\be,(\ga_1,\dots ,\ga_{n-2},-1)) .
    \end{align*}
  \item Since we have
    \begin{align*}
      M_{n-1}M_n \cdot \tf_{i_1 ,\dots ,i_{n-2},i_{n-1}}^{\pm} =
      \begin{cases}
        \tf_{i_1 ,\dots ,i_{n-2},0}^{\pm} & (i_{n-1}=0) \\
        \tf_{i_1 ,\dots ,i_{n-2},0}^{\pm}-\tf_{i_1 ,\dots ,i_{n-2},1}^{\pm} & (i_{n-1}=1),
      \end{cases}
    \end{align*}
    the representation matrix of the action by $M_{n-1}M_n$ is 
    \begin{align*}
      \underbrace{E_2 \otimes \cdots \otimes E_2}_{n-2} \otimes
      \begin{pmatrix} 1&1 \\ 0&-1 \end{pmatrix} 
      &=\underbrace{E_2 \otimes \cdots \otimes E_2}_{n-2} \otimes G(-1) \\
      &=M_{n-1}^{(n-1)}(\al,\be,(\ga_1,\dots ,\ga_{n-2},-1)). 
    \end{align*}
  \item Since we have
    \begin{align*}
      M_0 \cdot \tf_{i_1 ,\dots ,i_{n-2},i_{n-1}}^{+} =
      \begin{cases}
        \tf_{i_1 ,\dots ,i_{n-2},0}^{+}
        -\la_{0;i_1 ,\dots ,i_{n-2}}\tf_{1 ,\dots ,1,1}^{+} & (i_{n-1}=0) \\
        \tf_{i_1 ,\dots ,i_{n-2},1}^{+}
        -\la_{1;i_1 ,\dots ,i_{n-2}}\tf_{1 ,\dots ,1,1}^{+} & (i_{n-1}=1),
      \end{cases}
    \end{align*}
    the representation matrix of the action by $M_0$ is 
    \begin{align}
      \label{eq:rep-mat-M0-1}
      E_{2^{n-1}}-\tp (\mathbf{0}, \dots ,\mathbf{0},\bfv') ,
    \end{align}
    where the $I=(i_1,\dots ,i_{n-2},i_{n-1})$-th entries of $\bfv'$ are 
    \begin{align*}
      \begin{cases}
        \lambda_{0;i_1 ,\dots ,i_{n-2}} 
        =(-1)^n \frac{(\alpha-1)(\beta-1)\prod_{k=1}^{n-2} \gamma_k}{\alpha \beta}
        =(-1)^{n-1} \frac{(\alpha-1)(\beta-1)\prod_{k=1}^{n-2} \gamma_k\cdot (-1)}{\alpha \beta}
        & (i_1=\cdots=i_{n-1}=0) \\
        \lambda_{0;i_1 ,\dots ,i_{n-2}}
        =(-1)^{n+i_1+\dots +i_{n-2}}\frac{(\alpha \beta +(-1)^{i_1+\dots +i_{n-2}} 
          \prod_{k=1}^{n-2} \gamma_k^{i_k})\prod_{k=1}^{n-2} \gamma_k^{1-i_k}}{\alpha \beta}
        &  \\
        \qquad \qquad 
        =(-1)^{n-1+|I|}\frac{(\alpha \beta +(-1)^{|I|} 
          \prod_{k=1}^{n-2} \gamma_k^{i_k})\prod_{k=1}^{n-2} \gamma_k^{1-i_k} \cdot (-1)}{\alpha \beta}
        & (i_{n-1}=0, \exists i_k\neq 0) \\ 
        \lambda_{1;i_1 ,\dots ,i_{n-2}}
        =(-1)^{n+i_1+\dots +i_{n-2}}\frac{(\alpha \beta +(-1)^{i_1+\dots +i_{n-2}} 
          \prod_{k=1}^{n-2} \gamma_k^{i_k})\prod_{k=1}^{n-2} \gamma_k^{1-i_k}}{\alpha \beta}
        &  \\
        \qquad \qquad 
        =(-1)^{n-1+|I|}\frac{(\alpha \beta +(-1)^{|I|} 
          \prod_{k=1}^{n-2} \gamma_k^{i_k} \cdot (-1))\prod_{k=1}^{n-2} \gamma_k^{1-i_k}}{\alpha \beta}
        & (i_{n-1}=1) .
      \end{cases}
    \end{align*}
    Each of these entries coincides with $v^{(n-1)}_{I}(\al,\be,(\ga_1,\dots ,\ga_{n-2},-1))$. 
    Thus, the representation matrix (\ref{eq:rep-mat-M0-1}) is nothing but 
    $M_0^{(n-1)}(\al,\be,(\ga_1,\dots ,\ga_{n-2},-1))$. 
  \item We consider the action by $M_n M_0 M_n^{-1}$ on $W^{-}$. 
    Using 
    \begin{align*}
      M_n \cdot \tf_{i_1 ,\dots ,i_{n-2},i_{n-1}}^{\pm}=\tf_{i_1 ,\dots ,i_{n-2},i_{n-1}}^{\mp} , 
    \end{align*}
    we have 
    \begin{align*}
      M_n M_0 M_n^{-1} \cdot \tf_{i_1 ,\dots ,i_{n-2},i_{n-1}}^{-} =
      \begin{cases}
        \tf_{i_1 ,\dots ,i_{n-2},0}^{-}
        -\la_{0;i_1 ,\dots ,i_{n-2}}\tf_{1 ,\dots ,1,1}^{-} & (i_{n-1}=0) \\
        \tf_{i_1 ,\dots ,i_{n-2},1}^{-}
        -\la_{1;i_1 ,\dots ,i_{n-2}}\tf_{1 ,\dots ,1,1}^{-} & (i_{n-1}=1).
      \end{cases}
    \end{align*}
    Similarly to the aforementioned discussion, 
    we can show that the representation matrix of the action by $M_n M_0 M_n^{-1}$ 
    coincides with $M_0^{(n-1)}(\al,\be,(\ga_1,\dots ,\ga_{n-2},-1))$. 
  \end{itemize}
  Therefore, the subgroups
  $R^{\pm}$ are isomorphic to the smallest normal subgroup of 
  $\Mon^{(n-1)}(\al,\be,(\ga_1,\dots ,\ga_{n-2},-1))$, 
  which contains $M_0^{(n-1)}(\al,\be,(\ga_1,\dots ,\ga_{n-2},-1))$. 
  This is nothing but $\Ref^{(n-1)}(\al,\be,(\ga_1,\dots ,\ga_{n-2},-1))$. 
  Thus, the decomposition (\ref{eq:direct-prod-1}) implies the lemma. 
\end{proof}

\begin{proof}[Proof of Lemma \ref{lem:red-2}]
  Without loss of generality, 
  we may assume $\ga_n=\be\al^{-1}=-1$. 
  Note that $M_n^2 =E_{2^n}$. 
  For each $(i_1 ,\dots ,i_{n-1})\in \ztwo^{n-1}$, we set 
  \begin{align*}
    h_{i_1,\dots ,i_{n-1},0} 
    &
      % = \ee_{i_1,i_2,\dots ,i_{n-1},0} 
      =\ee_{i_1} \otimes \ee_{i_2} \otimes \cdots \otimes \ee_{i_{n-1}} \otimes \ee_0 , \\
    h_{i_1,\dots ,i_{n-1},1} 
    &
      % = 2\ee_{i_1,\dots ,i_{n-1},1}- \ee_{i_1,\dots ,i_{n-1},1} 
      =\ee_{i_1} \otimes \ee_{i_2} \otimes \cdots \otimes \ee_{i_{n-1}} \otimes (2\ee_1 -\ee_0) , \\
    f_{12;i_1 ,\dots ,i_{n-1}}^{\pm} &=h_{i_1,\dots ,i_{n-1},0} \pm h_{i_1,\dots ,i_{n-1},1}  ,
  \end{align*}
  and consider a direct sum decomposition of $\C^{2^n}$: 
  \begin{align*}
    \C^{2^n}=W^{+}\oplus W^{-} ;\quad 
    W^{\pm} = \bigoplus_{(i_1,\dots ,i_{n-1})} \C f_{12;i_1 ,\dots ,i_{n-1}}^{\pm} .
  \end{align*}
  The dimension of each factor is $2^{n-1}$. Note that 
  \begin{align*}
    \ee_{1,\dots ,1}
    =\frac{1}{2}(h_{1,\dots ,1,0}+h_{1,\dots ,1,1})
    =\frac{1}{2}f_{12;1,\dots ,1}^{+} \in W^{+} .
  \end{align*}
  As was shown in \cite{GK-FC-Zariski}, we obtain the following equalities: 
  \begin{align*}
    M_k \cdot f_{12; i_1,\dots ,i_k,\dots ,i_{n-1}}^{\pm}
    &=
    \begin{cases}
      f_{12; i_1,\dots, 0,\dots ,i_{n-1}}^{\pm} & (i_k=0) \\
      -\gamma_k^{-1} f_{12; i_1 ,\dots ,0,\dots ,i_{n-1}}^{\pm} 
      +\gamma_k^{-1} f_{12; i_1 ,\dots ,1,\dots,i_{n-1}}^{\pm} & (i_k=1),
    \end{cases}\\
    M_n \cdot f_{12;i_1 ,\dots ,i_{n-1}}^{\pm} 
    &= f_{12;i_1 ,\dots ,i_{n-1}}^{\mp} , \\
    M_0 \cdot f_{12;i_1,\dots ,i_{n-1}}^{-}
    &=f_{12;i_1,\dots ,i_{n-1}}^{-}\in W^{-} , \\ 
    M_0 \cdot f_{12;i_1,\dots ,i_{n-1}}^{+}
    &=f_{12;i_1,\dots ,i_{n-1}}^{+}+2\lambda_{i_1 ,\dots ,i_{n-1}}\ee_{1,\dots ,1} \in W^{+} , 
  \end{align*}
  where $1\leq k \leq n-1$ and 
  \begin{align*}
    \lambda_{i_1 ,\dots ,i_{n-1}} =(-1)^{n+i_1+\dots +i_{n-1}}
    \frac{(\alpha \beta +(-1)^{i_1+\dots +i_{n-1}} 
      \prod_{k=1}^{n-1} \gamma_k^{i_k})\prod_{k=1}^{n-1} \gamma_k^{1-i_k}}{\alpha \beta} .
  \end{align*}
  These equalities imply that $W^{\pm}$ are non-trivial $\Ref$-subspaces 
  (see \cite{GK-FC-Zariski}). 
  Because
  \begin{itemize}
  \item $M_0$ acts trivially on $W^{-}$,  
  \item $M_n M_0 M_n^{-1}$ acts trivially on $W^{+}$, 
  \end{itemize}
  we have the direct product decomposition
  \begin{align}
    \label{eq:direct-prod-2}
    &\Ref^{(n)}(\al,\be,(\ga_1,\dots ,\ga_{n-1},-1))\\
    \nonumber
    &=\left\langle 
      gM_0 g^{-1},gM_n M_0 M_n^{-1} g^{-1}
      \mid 
      g=M_1^{j_1} \cdots M_{n-1}^{j_{n-1}} ,\ j_k\in \Z
      \right\rangle \\
    \nonumber
    &=\left\langle 
      gM_0 g^{-1}
      \mid 
      g=M_1^{j_1} \cdots M_{n-1}^{j_{n-1}} ,\ j_k\in \Z
      \right\rangle 
      \times \left\langle 
      g M_n M_0 M_n^{-1} g^{-1}
      \mid 
      g=M_1^{j_1} \cdots M_{n-1}^{j_{n-1}} ,\ j_k\in \Z
      \right\rangle \\
    \nonumber
    &=R^{+} \times R^{-} .
  \end{align}
  Here, $R^{+}$ (resp. $R^{-}$) acts trivially on $W^{-}$ (resp. $W^{+}$). 
  We consider the representation matrices of the actions by 
  $M_k$ ($1\leq k \leq n-1$) on $W^{\pm}$, 
  $M_0$ on $W^{+}$, and $M_n M_0 M_n^{-1}$ on $W^{-}$
  with respect to the bases $\{ f_{12;i_1 ,\dots ,i_{n-1}}^{\pm} \}$. 
  \begin{itemize}
  \item Similarly to the proof of Lemma \ref{lem:red-1}, 
    the representation matrix of the action by $M_k$ ($1\leq k \leq n-1$) is
    $M_k^{(n-1)}(\al,\be,(\ga_1,\dots ,\ga_{n-2},\ga_{n-1}))$. 
  \item We consider the action by $M_0$ on $W^{+}$. 
    Since      
    \begin{align*}
      M_0 \cdot f_{12;i_1 ,\dots ,i_{n-1}}^{+}
      = f_{12;i_1,\dots ,i_{n-1}}^{+}
      +\lambda_{i_1 ,\dots ,i_{n-1}}f_{12;1,\dots ,1}^{+} ,
    \end{align*}
    the representation matrix is 
    \begin{align}
      \label{eq:rep-mat-M0-2}
      E_{2^{n-1}}-{}^t (\mathbf{0}, \dots ,\mathbf{0},\bfv') ,
    \end{align}
    where the $I=(i_1,\dots ,i_{n-2},i_{n-1})$-th entry of $\bfv'$ is 
    \begin{align*}
      -\lambda_{i_1 ,\dots ,i_{n-1}}
      =(-1)^{n-1+i_1+\dots +i_{n-1}}
      \frac{(\alpha \beta +(-1)^{i_1+\dots +i_{n-1}} 
      \prod_{k=1}^{n-1} \gamma_k^{i_k})\prod_{k=1}^{n-1} \gamma_k^{1-i_k}}{\alpha \beta} .
    \end{align*}
    This entry is nothing but $v^{(n-1)}_I (\al,\be,(\ga_1,\dots ,\ga_{n-2},\ga_{n-1}))$ 
    if $I\neq (0 ,\dots ,0)$. 
    Because we have $\alpha +\beta=0$ by the assumption of the lemma, 
    the $(0 ,\dots ,0)$-th entry is written as 
    \begin{align*}
      -\lambda_{0 ,\dots ,0} 
      &=-(-1)^{n}\frac{(\alpha \beta +1)\prod_{k=1}^{n-1} \gamma_k}{\alpha \beta}
        =(-1)^{n-1} \frac{(\alpha -1)(\beta -1)\prod_{k=1}^{n-1} \gamma_k}{\alpha \beta} \\
      &=v^{(n-1)}_{0,\dots ,0}(\al,\be,(\ga_1,\dots ,\ga_{n-2},\ga_{n-1})) .
    \end{align*}
    % Thus, 
    This implies that 
    the representation matrix (\ref{eq:rep-mat-M0-2}) coincides with 
    the reflection 
    $M_0^{(n-1)}(\al,\be,(\ga_1,\dots ,\ga_{n-2},\ga_{n-1}))$. 
  \item Similarly to the proof of Lemma \ref{lem:red-1}, 
    we can show that the representation matrix 
    of the action by $M_n M_0 M_n^{-1}$ on $W^{-}$ is also 
    $M_0^{(n-1)}(\al,\be,(\ga_1,\dots ,\ga_{n-2},\ga_{n-1}))$. 
  \end{itemize}
  Therefore, the subgroups 
  $R^{\pm}$ are isomorphic to the reflection subgroup 
  $\Ref^{(n-1)}(\al,\be,(\ga_1,\dots ,\ga_{n-2},\ga_{n-1}))$; 
  the decomposition (\ref{eq:direct-prod-2}) implies the lemma. 
\end{proof}

From the proofs of Lemmas \ref{lem:red-1} and \ref{lem:red-2}, 
$\Mon^{(n)}(\al ,\be ,\ga)$ is imprimitive (see \cite[Definition 5.1]{BH}) 
if at least two of 
$\ga_1,\dots ,\ga_n ,\al \be^{-1}$ are $-1$. 
For $n\geq 3$, 
if $\Mon^{(n)}(\al ,\be ,\ga)$ is finite irreducible, 
then the condition (\ref{condition-B}) in Theorem \ref{th:fin-irr} 
implies that at least two of $\ga_1,\dots ,\ga_n ,\al \be^{-1}$ are $-1$. 
Thus, we obtain the following as a corollary of Theorem \ref{th:fin-irr}. 

\begin{Cor}
  We assume $n\geq 3$. 
  If the monodromy group $\Mon^{(n)}(\al,\be,\ga)$ is finite irreducible, 
  then it is imprimitive. 
\end{Cor}

\section{Proof of Theorem \ref{th:fin-irr}}
\subsection{Proof of ``if'' part}
We assume the conditions (\ref{condition-A}) and (\ref{condition-B}) in 
Theorem \ref{th:fin-irr}. 
When we assume the condition (\ref{condition-B}), 
it suffices to consider the following four cases 
without loss of generality:
\begin{enumerate}[(\ref{condition-B}-a)]
\item\label{condition-B-a} 
  $\ga_1=\cdots =\ga_n=-1$, 
\item\label{condition-B-b} 
  $\ga_2=\cdots =\ga_n=\be\al^{-1}=-1$, 
\item\label{condition-B-c} 
  $\ga_2=\cdots =\ga_n=\de_0^{(n)}(\al,\be,\ga)=-1$, 
\item\label{condition-B-d} 
  $\ga_3=\cdots =\ga_n=\be\al^{-1}=\de_0^{(n)}(\al,\be,\ga)=-1$.
\end{enumerate}

\begin{Lem}
  \label{lem:if-irred}
  If the conditions (\ref{condition-A}) and (\ref{condition-B}) hold, 
  then the irreducibility condition (\ref{irred-2}) holds, and hence 
  % then we have
  % \begin{align}
  %   \label{eq:irred-2}
  %   \alpha-\prod_{k=1}^n \gamma_k^{i_k} ,\quad 
  %   \beta-\prod_{k=1}^n \gamma_k^{i_k} \neq 0 ,\qquad 
  %   \forall I=(i_1,\ldots,i_n) \in \{0,1\}^n.
  % \end{align}
  % Since these are equivalent to (\ref{irred-1}), 
  the monodromy group $\Mon^{(n)}(\al,\be,\ga)$ is irreducible. 
\end{Lem}
\begin{proof}
  Since $\Mon^{(1)}(\al,\be,\ga_k)$ is irreducible 
  by the condition (\ref{condition-A}), we have 
  \begin{align}
    \label{eq:Gauss-irred}
    \al-1,\ \al-\ga_k ,\ \be-1, \ \be-\ga_k \neq 0 
  \end{align}
  for $k=1 ,\dots ,n$. 
  We consider the four cases 
  (\ref{condition-B}-\ref{condition-B-a})--(\ref{condition-B}-\ref{condition-B-d}). 
  \begin{itemize}
  \item[(\ref{condition-B}-\ref{condition-B-a})] 
    In this case, the condition (\ref{irred-2}) is reduced to 
    $\al\pm 1\neq 0$ and $\be\pm 1 \neq 0$. 
    These are nothing but (\ref{eq:Gauss-irred}) for $k=1$. 
  \item[(\ref{condition-B}-\ref{condition-B-b})] 
    The non-trivial conditions in (\ref{irred-2}) are 
    $\al +\ga_1 \neq 0$ and $\be +\ga_1 \neq 0$. 
    By $\be\al^{-1}=-1$, these are equivalent to 
    $-\be +\ga_1 \neq 0$ and $-\al +\ga_1 \neq 0$, respectively. 
    The last conditions follow from (\ref{eq:Gauss-irred}). 
  \item[(\ref{condition-B}-\ref{condition-B-c})] 
    Similarly to the case (\ref{condition-B}-\ref{condition-B-b}), 
    the non-trivial conditions in (\ref{irred-2}) are 
    $\al +\ga_1 \neq 0$ and $\be +\ga_1 \neq 0$. Because of the identity
    \begin{align*}
      -1 =\de_0^{(n)}(\al,\be,\ga)=(-1)^{n-1}\ga_1\cdots\ga_n\al^{-1}\be^{-1} 
      =\ga_1 \al^{-1} \be^{-1}, 
    \end{align*}
    we have $\ga_1=-\al \be$,  
    hence, we obtain $\al +\ga_1 =-\al (\be-1)\neq 0$ and 
    $\be +\ga_1 =-\be (\al-1)\neq 0$. 
  \item[(\ref{condition-B}-\ref{condition-B-d})] 
    In this case, the non-trivial conditions in (\ref{irred-2}) are 
    \begin{align*}
      \al +\ga_1 \neq 0,\quad 
      \al +\ga_2 \neq 0 ,\quad 
      \al -\ga_1\ga_2 \neq 0 ,\quad 
      \al +\ga_1\ga_2 \neq 0 ,
    \end{align*}
    % and those replaced $\al$ with $\be$. 
    and those obtained by replacing $\al$ with $\be$. 
    Because of the identity
    \begin{align*}
      -1 =\de_0^{(n)}(\al,\be,\ga)=(-1)^{n-1}\ga_1\cdots\ga_n\al^{-1}\be^{-1} 
      =-\ga_1\ga_2 \al^{-1} \be^{-1}
    \end{align*}
    and $\be\al^{-1}=-1$, we obtain 
    \begin{align*}
      \al +\ga_1 
      &=\al+\al \be \ga_2^{-1} =\al\ga_2^{-1}(\ga_2 +\be)
        =-\al\ga_2^{-1}(\al -\ga_2) \neq 0 , \\
      \al -\ga_1\ga_2 
      &=\al - \al \be =\al(1- \be)=-\al(\be-1)\neq 0, \\
      \al +\ga_1\ga_2 
      &=\al + \al \be =\al(1+ \be)=\al(1-\al)=-\al(\al-1)\neq 0 .
    \end{align*}
  \end{itemize}
\end{proof}

\begin{Prop}
  \label{prop:if-fin}
  If the conditions (\ref{condition-A}) and (\ref{condition-B}) hold, then 
  the reflection subgroup $\Ref^{(n)}(\al,\be,\ga)$ is finite. 
\end{Prop}
\begin{proof}
  By Lemma \ref{lem:if-irred}, we may assume that 
  the irreducibility condition (\ref{irred-2}) holds. 
  % We show finiteness of the reflection subgroup $\Ref$. 
  Thus, we can apply Lemmas \ref{lem:red-1} and \ref{lem:red-2}. 
  Let us consider the four cases 
  (\ref{condition-B}-\ref{condition-B-a})--(\ref{condition-B}-\ref{condition-B-d}). 
  \begin{itemize}
  \item[(\ref{condition-B}-\ref{condition-B-a})] 
    Using Lemma \ref{lem:red-1} repeatedly, we have 
    \begin{align*}
      \Ref^{(n)}(\al,\be,(-1,-1,-1,\dots ,-1)) 
      &\simeq 
        \left(\Ref^{(n-1)}(\al,\be,(-1,\dots ,-1,-1)) \right)^2 \\
      &\simeq 
        \left(\Ref^{(n-2)}(\al,\be,(-1,\dots ,-1)) \right)^4 \\
      &\simeq \cdots 
        \simeq 
        \left(\Ref^{(1)}(\al,\be,-1) \right)^{2^{n-1}} .
    \end{align*}
    The finiteness follows from the condition (\ref{condition-A}). 
  \item[(\ref{condition-B}-\ref{condition-B-b})] 
    Using Lemma \ref{lem:red-2} repeatedly, we have
    \begin{align*}
      \Ref^{(n)}(\al,\be,(\ga_1, -1,\dots ,-1)) 
      \simeq 
      \left(\Ref^{(1)}(\al,\be,\ga_1) \right)^{2^{n-1}} . 
    \end{align*}
    The finiteness follows from the condition (\ref{condition-A}).     
  \item[(\ref{condition-B}-\ref{condition-B-c})] 
    Using Lemma \ref{lem:red-1} repeatedly, we obtain 
    \begin{align*}
      \Ref^{(n)}(\al,\be,(\ga_1 ,-1,\dots ,-1)) 
      \simeq 
      \left(\Ref^{(2)}(\al,\be,(\ga_1, -1)) \right)^{2^{n-2}} . 
    \end{align*}
    Since 
    \begin{align*}
      &\de_0^{(2)}(\al,\be,(\ga_1, -1))
      =-\ga_1 \cdot (-1) \cdot \al^{-1}\be^{-1} \\
      &=(-1)^{n-1}\ga_1\cdots\ga_n\al^{-1}\be^{-1} 
      =\de_0^{(n)}(\al,\be,\ga)=-1 ,
    \end{align*}
    $\Ref^{(2)}(\al,\be,(\ga_1, -1))$ is finite by 
    Fact \ref{fact:F4-fin-irr}, 
    and hence, $\Ref^{(n)}(\al,\be,(\ga_1 ,-1,\dots ,-1))$ is also finite. 
  \item[(\ref{condition-B}-\ref{condition-B-d})] 
    Using Lemma \ref{lem:red-2} repeatedly, we obtain 
    \begin{align*}
      \Ref^{(n)}(\al,\be,(\ga_1,\ga_2,-1,\dots ,-1)) 
      \simeq 
      \left(\Ref^{(2)}(\al,\be,(\ga_1,\ga_2)) \right)^{2^{n-2}}  .
    \end{align*}
    Because of the identity 
    \begin{align*}
      &\de_0^{(2)}(\al,\be,(\ga_1,\ga_2))=-\ga_1\ga_2 \al^{-1}\be^{-1} \\
      &=(-1)^{n-1}\ga_1\cdots\ga_n\al^{-1}\be^{-1} 
      =\de_0^{(n)}(\al,\be,\ga)=-1
    \end{align*}
    and Fact \ref{fact:F4-fin-irr}, $\Ref^{(2)}(\al,\be,(\ga_1, \ga_2))$ is finite. 
    Thus, 
    we can conclude that 
    $\Ref^{(n)}(\al,\be,(\ga_1,\ga_2,-1,\dots ,-1))$ is also finite. 
  \end{itemize}
\end{proof}

Using Fact \ref{fact-ref}, Lemma \ref{lem:if-irred}, and Proposition \ref{prop:if-fin}, 
we complete the proof of the ``if'' part of Theorem \ref{th:fin-irr}. 

% \begin{Rem}
%   According to the proof of Proposition \ref{prop:if-fin}, 
%   when $\Mon^{(n)}(\al,\be,\ga)$ is finite irreducible,  
%   the reflection subgroup $\Ref^{(n)}(\al,\be,\ga)$ is decomposed into 
%   the product of some
%   $\Ref^{(1)}$'s or $\Ref^{(2)}$'s. 
%   Their structures were studied in \cite{Kato} and \cite{Kato-Sekiguchi}. 
% \end{Rem}

\subsection{Proof of ``only if'' part}
We may assume that the irreducibility condition (\ref{irred-2}) holds. 
First, we consider the condition (\ref{condition-A}).
\begin{Lem}
  Let $\BFH$ be the subgroup of $\Mon^{(n)}(\al,\be,\ga)$ generated by 
  $M_1$, $M_2 ,\dots ,M_{n-1}$ and $M_0 M_n M_0$. 
  Then, there exists a surjective group homomorphism  
  $\BFH \to \Mon^{(n-1)}(\al,\be,(\ga_1,\dots ,\ga_{n-1}))$.  
  % There is a subgroup of $\Mon^{(n)}(\al,\be,\ga)$ which is isomorphic to 
  % $\Mon^{(n-1)}(\al,\be,(\ga_1,\dots ,\ga_{n-1}))$.  
  % More precisely, we have 
  % \begin{align*}
  %   % \label{eq:dim-decrease}
  %   \langle M_0 M_n M_0, M_1,M_2 ,\dots ,M_{n-1} \rangle 
  %   \simeq \Mon^{(n-1)}(\al,\be,(\ga_1,\dots ,\ga_{n-1})) .
  % \end{align*}
\end{Lem}
\begin{proof}
  We consider a subspace
  \begin{align*}
    W=\bigoplus_{(i_1 ,\dots ,i_{n-1})} \C \ee_{i_1,\dots ,i_{n-1},0}
  \end{align*}
  of $\C^{2^n}$ of which the dimension is $2^{n-1}$. 
  We prove the following two claims. 
  % which imply (\ref{eq:dim-decrease}). 
  \begin{enumerate}[(i)]
  \item For $k=1 ,\dots, n-1$, $M_k$ acts on $W$ and 
    its representation matrix coincides with 
    $M_k^{(n-1)}(\al,\be,(\ga_1,\dots ,\ga_{n-1}))$. 
  \item $M_0 M_n M_0$ acts on $W$ and 
    its representation matrix coincides with 
    $M_0^{(n-1)}(\al,\be,(\ga_1,\dots ,\ga_{n-1}))$. 
  \end{enumerate}
  % If we prove them, then we can see
  Proving these claims would confirm 
  that each element of $\BFH$ is of the form
  \begin{align*}
    \begin{pmatrix}
      H & H' \\ O & H''
    \end{pmatrix} ,\quad
    \begin{array}{l}
      H \in \Mon^{(n-1)}(\al,\be,(\ga_1,\dots ,\ga_{n-1})) ,\\
      \text{$H'$ and $H''$ are square matrices of size $2^{n-1}$.}
    \end{array}
  \end{align*}
  Thus, we can define the group homomorphism
  \begin{align*}
    \BFH \longrightarrow \Mon^{(n-1)}(\al,\be,(\ga_1,\dots ,\ga_{n-1}));\qquad
    \begin{pmatrix}
      H & H' \\ O & H''
    \end{pmatrix}
    \mapsto H, 
  \end{align*}
  and it is clearly surjective. 
  % a correspondence
  % \begin{align*}
  %   \Mon^{(n-1)}(\al,\be,(\ga_1,\dots ,\ga_{n-1}))
  %     &\to \langle M_0 M_n M_0, M_1,M_2 ,\dots ,M_{n-1} \rangle ; \\
  %   M_0^{(n-1)}(\al,\be,(\ga_1,\dots ,\ga_{n-1})) &\mapsto M_0 M_n M_0 ,\\ 
  %   M_k^{(n-1)}(\al,\be,(\ga_1,\dots ,\ga_{n-1})) &\mapsto M_k 
  %     \quad (k=1,\dots ,n-1)
  % \end{align*}
  % gives an isomorphism. 
  First, we present the proof of the claim (i). Since the $n$th factor of
  \begin{align}
    \label{eq:Mk-tensor}
    M_k =E_2 \ot \cdots \ot E_2 \ot G(\ga_k) \ot E_2 \ot \cdots \ot \underline{E_2} ,
  \end{align}
  is $E_2$ (underlined), that of 
  $\ee_{i_1,\dots ,i_{n-1},0} 
  =\ee_{i_1} \otimes \cdots \otimes \ee_{i_{n-1}} \otimes \underline{\ee_0}$
  is not changed. This means that $M_k$ acts on $W$. 
  Its representation matrix is obtained by removing 
  the $n$th factor $E_2$ from (\ref{eq:Mk-tensor}). 
  This is nothing but $M_k^{(n-1)}(\al,\be,(\ga_1,\dots ,\ga_{n-1}))$, 
  and the claim (i) is proved. 
  Next, we provide the proof of the claim (ii). 
  % By using the notation $v_I$ mentioned below Fact \ref{fact:rep-mat}, 
  We have
  \begin{align*}
    &M_0 M_n M_0 \ee_{i_1,\dots ,i_{n-1},0}
    =M_0 M_n (\ee_{i_1,\dots ,i_{n-1},0} -v^{(n)}_{i_1,\dots ,i_{n-1},0}\ee_{1,\dots,1}) \\
    &=M_0 (\ee_{i_1,\dots ,i_{n-1},0} 
      -v^{(n)}_{i_1,\dots ,i_{n-1},0}(-\ga_n^{-1}\ee_{1,\dots,1,0} +\ga_n^{-1}\ee_{1,\dots,1})) \\
    &=\ee_{i_1,\dots ,i_{n-1},0} -v^{(n)}_{i_1,\dots ,i_{n-1},0}\ee_{1,\dots,1}\\
    &\qquad 
      +v^{(n)}_{i_1,\dots ,i_{n-1},0}\ga_n^{-1}(\ee_{1,\dots ,1,0} -v^{(n)}_{1,\dots ,1,0}\ee_{1,\dots,1})
    % \\
    % &\qquad 
      -v^{(n)}_{i_1,\dots ,i_{n-1},0}\ga_n^{-1}(\ee_{1,\dots ,1} -v^{(n)}_{1,\dots ,1}\ee_{1,\dots,1}) \\
    &=\ee_{i_1,\dots ,i_{n-1},0} +v^{(n)}_{i_1,\dots ,i_{n-1},0}\ga_n^{-1}\ee_{1,\dots ,1,0}
    % \\
    % &\qquad 
      -v^{(n)}_{i_1,\dots ,i_{n-1},0}(1+\ga_n^{-1}v^{(n)}_{1,\dots ,1,0}
      +\ga_n^{-1}(1-v^{(n)}_{1,\dots ,1})) \ee_{1,\dots,1} .
  \end{align*}
  Because of the identity 
  \begin{align*}
    &1+\ga_n^{-1}v^{(n)}_{1,\dots ,1,0}+\ga_n^{-1}(1-v^{(n)}_{1,\dots ,1})
      =1+\ga_n^{-1}v^{(n)}_{1,\dots ,1,0}+\ga_n^{-1}\de_0^{(n)}(\al,\be,\ga)  \\
    &=1+\ga_n^{-1}\cdot (-1)\cdot \frac{(\alpha \beta +(-1)^{n-1} \prod_{k=1}^{n-1} \gamma_k)
      \gamma_n}{\alpha \beta}
      +\ga_n^{-1}\cdot (-1)^{n-1} \frac{\prod_{k=1}^n \gamma_k}{\al \be} =0 ,
  \end{align*}
  we obtain 
  \begin{align*}
    M_0 M_n M_0 \ee_{i_1,\dots ,i_{n-1},0}
    =\ee_{i_1,\dots ,i_{n-1},0} +v^{(n)}_{i_1,\dots ,i_{n-1},0}\ga_n^{-1}\ee_{1,\dots ,1,0} 
    \in W,
  \end{align*}
  and hence, $M_0 M_n M_0$ acts on $W$. 
  The representation matrix is equal to
  \begin{align}
    \label{eq:rep-mat-M0MnM0}
    E_{2^{n-1}}-\tp (\mathbf{0}, \dots ,\mathbf{0},\bfv'), 
  \end{align}
  where the $(i_1,\dots ,i_{n-1})$-th entry $v'_{i_1,\dots ,i_{n-1}}$ of $\bfv'$ is 
  $v'_{i_1,\dots ,i_{n-1}} =-v^{(n)}_{i_1,\dots ,i_{n-1},0}\cdot \ga_n^{-1}$. 
  If $(i_1,\dots ,i_{n-1})=(0,\dots ,0)$, then we have
  \begin{align*}
    v'_{0,\dots ,0}
    &=-v^{(n)}_{0,\dots ,0} \cdot \ga_n^{-1}
    =-(-1)^n \frac{(\alpha-1)(\beta-1)\prod_{k=1}^n \gamma_k}{\alpha \beta} \cdot \ga_n^{-1} \\
    &=(-1)^{n-1} \frac{(\alpha-1)(\beta-1)\prod_{k=1}^{n-1} \gamma_k}{\alpha \beta} 
      =v^{(n-1)}_{0,\dots ,0}(\al,\be,(\ga_1,\dots ,\ga_{n-1})).
  \end{align*}
  Otherwise, we have
  \begin{align*}
    &v'_{i_1,\dots ,i_{n-1}}
      =-v^{(n)}_{i_1,\dots ,i_{n-1},0}\cdot \ga_n^{-1} \\
    &=-(-1)^{n+|I|} \frac{(\alpha \beta +(-1)^{|I|} \prod_{k=1}^{n-1} \gamma_k^{i_k})
      \prod_{k=1}^{n-1} \gamma_k^{1-i_k} \ga_n}{\alpha \beta} \cdot \ga_n^{-1} \\
    &=(-1)^{n-1+|I|} \frac{(\alpha \beta +(-1)^{|I|} \prod_{k=1}^{n-1} \gamma_k^{i_k})
      \prod_{k=1}^{n-1} \gamma_k^{1-i_k}}{\alpha \beta} \\
    &=v^{(n-1)}_{i_1,\dots ,i_{n-1}}(\al,\be,(\ga_1,\dots ,\ga_{n-1})).
  \end{align*}
  Therefore, the representation matrix (\ref{eq:rep-mat-M0MnM0}) coincides with 
  $M_0^{(n-1)}(\al,\be,(\ga_1,\dots ,\ga_{n-1}))$, and 
  the proof is completed. 
\end{proof}

Using this lemma, we obtain the following corollary. 
\begin{Cor}
  \label{cor:monodromy-subgroup}
  For $1\leq j_1 <j_2 <\cdots <j_k \leq n$ ($k=1,\dots ,n$), 
  $\Mon^{(k)}(\al,\be,(\ga_{j_1},\dots ,\ga_{j_k}))$ is isomorphic to
  a subquotient group of $\Mon^{(n)}(\al,\be,\ga)$. 
  % there is a subgroup of $\Mon^{(n)}(\al,\be,\ga)$ which is isomorphic to 
  % $\Mon^{(k)}(\al,\be,(\ga_{j_1},\dots ,\ga_{j_k}))$. 
\end{Cor}
Let us show that the condition (\ref{condition-A}) holds. 
\begin{Prop}
  \label{prop:k-variables}
  Suppose that $\Mon^{(n)}(\al,\be,\ga)$ is finite irreducible.  
  Then, for $1\leq j_1 <j_2 <\cdots <j_k \leq n$ ($k=1,\dots ,n$), 
  $\Mon^{(k)}(\al,\be,(\ga_{j_1},\dots ,\ga_{j_k}))$ is also finite irreducible.  
  Especially, the condition (\ref{condition-A}) holds. 
\end{Prop}
\begin{proof}
  The irreducibility of 
  $\Mon^{(k)}(\al,\be,(\ga_{j_1},\dots ,\ga_{j_k}))$ immediately follows from that of $\Mon^{(n)}(\al,\be,\ga)$ (recall that 
  the irreducibility condition is given by (\ref{irred-2})). 
  The finiteness follows from Corollary \ref{cor:monodromy-subgroup}.
\end{proof}

Next, we consider the condition (\ref{condition-B}).
\begin{Lem}
  \label{lem:F4-fin-1}
  Let $n\geq 3$ and $\Mon^{(n)}(\al,\be,\ga)$ be finite irreducible. 
  For distinct $i,j,k \in \{1,\dots ,n\}$, 
  $\Mon^{(2)}(\al,\al\ga_k^{-1},(\ga_i,\ga_j))$ and $\Mon^{(2)}(\be,\be\ga_k^{-1},(\ga_i,\ga_j))$ 
  are also finite irreducible. 
\end{Lem}
\begin{proof}
  For simplicity, we prove the claim only for $\Mon^{(2)}(\al,\al\ga_n^{-1},(\ga_{n-2},\ga_{n-1}))$.  
  As was mentioned in \cite[Remark 5.10]{G-FC-monodromy}, 
  $x_n^{-a} f(\frac{x_1}{x_n},\ldots ,\frac{x_{n-1}}{x_n},\frac{1}{x_n})$ 
  is a solution to $E_C (a,b,c)$ 
  if and only if $f(\xi_1,\ldots ,\xi_n)$ is a solution to 
  $E_C (a,a-c_n+1,(c_1,\ldots ,c_{n-1},a-b+1))$ 
  with variables $\xi_1 ,\ldots ,\xi_n$. 
  Then, the finiteness of $\Mon^{(n)}(\al,\be,\ga)$ 
  implies that of $\Mon^{(n)}(\al,\al\ga_n^{-1},(\ga_1,\ldots ,\ga_{n-1},\al\be^{-1}))$.
  Using Proposition \ref{prop:k-variables} with 
  $(j_1 ,j_2)=(n-2,n-1)$, 
  % $(j_1 ,j_2)=(n-2,n-1)$ ($k=2$), 
  we conclude that 
  $\Mon^{(2)}(\al,\al\ga_n^{-1},(\ga_{n-2},\ga_{n-1}))$ is finite irreducible. 
\end{proof}

\begin{Lem}
  \label{lem:gamma(-1)}
  For $n \geq 3$, if $\Mon^{(n)}(\al,\be,\ga)$ is finite irreducible, 
  then at least $n-2$ of $\ga_1,\dots,\ga_n$ are $-1$.  
\end{Lem}
\begin{proof}
  When the number of $k$ such that $\ga_k \neq -1$ is at most one, 
  the claim holds. We assume that $\ga_1 \neq -1$, $\ga_2 \neq -1$, and 
  show that $\ga_k=-1$ ($k=3,\dots,n$). 
  By Lemma \ref{lem:F4-fin-1}, 
  $\Mon^{(2)}(\al,\al\ga_k^{-1},(\ga_1,\ga_2))$ and $\Mon^{(2)}(\be,\be\ga_k^{-1},(\ga_1,\ga_2))$
  are finite irreducible. 
  By Fact \ref{fact:F4-fin-irr} (\ref{condition-B}') for 
  $\Mon^{(2)}(\al,\al\ga_k^{-1},(\ga_1,\ga_2))$, we have two possibilities: 
  \begin{enumerate}[(i)]
  \item $\de_0^{(2)}(\al,\al\ga_k^{-1},(\ga_1,\ga_2))=-1$; 
  \item at least two of 
    $\ga_1,\ga_2,\al\ga_k^{-1}\cdot \al^{-1}=\ga_k^{-1}$ are $-1$. 
  \end{enumerate}
  By the assumption, (ii) does not occur and we obtain 
  \begin{align*}
    1=-\de_0^{(2)}(\al,\al\ga_k^{-1},(\ga_1,\ga_2))=-(-\ga_1\ga_2\al^{-1}(\al\ga_k^{-1})^{-1})
    =\ga_1\ga_2\ga_k \al^{-2} .
  \end{align*}
  Similarly, we obtain $\ga_1\ga_2\ga_k \be^{-2} =1$ 
  from the finiteness of $\Mon^{(2)}(\be,\be\ga_k^{-1},(\ga_1,\ga_2))$. 
  Thus, we have 
  \begin{align*}
    1=(\ga_1\ga_2\ga_k \al^{-2})(\ga_1\ga_2\ga_k \be^{-2})
    =\left( \de_0^{(2)}(\al,\be,(\ga_1,\ga_2)) \right)^2 \ga_k^2 .
  \end{align*}
  On the other hand, Proposition \ref{prop:k-variables} implies that 
  $\Mon^{(2)}(\al,\be,(\ga_1,\ga_2))$ is also finite irreducible. 
  By Fact \ref{fact:F4-fin-irr} (\ref{condition-B}') and the assumption, 
  we have $\de_0^{(2)}(\al,\be,(\ga_1,\ga_2))=-1$. 
  Therefore, we obtain $\ga_k^2=1$, that is, $\ga_k =1$ or $\ga_k=-1$. 
  Because the matrix $G(1)=\begin{pmatrix}1 & -1 \\ 0 & 1 \end{pmatrix}$ has 
  infinite order, the matrix $M_k$ also has infinite order if $\ga_k=1$. 
  This is a contradiction. 
  Therefore, we conclude that $\ga_k=-1$. 
\end{proof}

By the following proposition, we complete the proof of Theorem \ref{th:fin-irr}. 
\begin{Prop}
  For $n \geq 3$, if $\Mon^{(n)}(\al,\be,\ga)$ is finite irreducible, 
  then the condition (\ref{condition-B}) holds. 
\end{Prop}
\begin{proof}
  By Lemma \ref{lem:gamma(-1)}, we may assume that
  \begin{align}
    \label{eq:gamma3-n=-1}
    \ga_3=\cdots=\ga_n=-1 
  \end{align}
  without loss of generality. 
  % By using Lemma \ref{lem:red-2} repeatedly, we obtain 
  % \begin{align*}
  %   \Ref^{(n)}(\al,\be,(\ga_1,\ga_2,-1,\dots ,-1)) 
  %   \simeq 
  %   \left(\Ref^{(3)}(\al,\be,(\ga_1,\ga_2,-1)) \right)^{2^{n-3}} ,
  % \end{align*}
  Proposition \ref{prop:k-variables} implies that 
  $\Mon^{(2)}(\al,\be,(\ga_1,\ga_2))$ and 
  $\Mon^{(2)}(\al,\be,(\ga_i,-1))$ ($i=1,2$) are finite irreducible. 
  By Fact \ref{fact:F4-fin-irr} (\ref{condition-B}') for 
  $\Mon^{(2)}(\al,\be,(\ga_1,\ga_2))$, we have two possibilities:
  \begin{enumerate}[(i)]
  \item $\de_0^{(2)}(\al ,\be,(\ga_1 ,\ga_2))=-1$, that is, 
    $-\ga_1 \ga_2 \al^{-1}\be^{-1}=-1$; 
  \item at least two of $\ga_1 ,\ga_2$, $\be\al^{-1}$ are $-1$. 
  \end{enumerate}
  In the case (ii), the condition (\ref{eq:gamma3-n=-1}) implies 
  (\ref{condition-B}-\ref{condition-B-a}) or (\ref{condition-B}-\ref{condition-B-b}), 
  and the proposition holds. 
  We consider the case when (ii) does not hold. 
  We may assume $\ga_1 \neq -1$. 
  Since (i) holds, we have
  \begin{align*}
    \de_0^{(n)}(\al,\be,\ga)
    =(-1)^{n-1}\ga_1\ga_2\ga_3\cdots\ga_n\al^{-1}\be^{-1} 
    =-\ga_1 \ga_2 \al^{-1}\be^{-1}
    =-1 .
  \end{align*}
  Therefore, if $\be\al^{-1}= -1$, then 
  the condition (\ref{condition-B}-\ref{condition-B-d}) holds. 
  We assume $\be\al^{-1}\neq -1$ and show $\ga_2 =-1$, which implies 
  (\ref{condition-B}-\ref{condition-B-c}). 
  By Fact \ref{fact:F4-fin-irr} (\ref{condition-B}') for 
  $\Mon^{(2)}(\al,\be,(\ga_1,-1))$, we have two possibilities:
  \begin{enumerate}[(i)]
    \setcounter{enumi}{2}
  \item $\de_0^{(2)}(\al ,\be,(\ga_1 ,-1))=-1$, that is, 
    $\ga_1 \al^{-1}\be^{-1}=-1$; 
  \item at least two of $\ga_1 ,-1$, $\be\al^{-1}$ are $-1$. 
  \end{enumerate}
  By the assumption, (iv) does not hold and we obtain $\ga_1 \al^{-1}\be^{-1}=-1$. 
  This and (i) imply
  \begin{align*}
    \ga_2 =(\ga_1 \al^{-1}\be^{-1})^{-1} =-1 .
  \end{align*}
  Therefore, we complete the proof. 
\end{proof}

\section{Structure of the finite irreducible monodromy group}\label{section-str}
In this section, we consider the structure of $\Mon^{(n)}$ when
it is finite irreducible.
Note that $\al,\be ,\ga_1 ,\dots ,\ga_n$ are roots of unity 
(equivalently, $a,b,c_1,\dots ,c_n \in \Q$) by the condition (A). 
For $q\in \{ 3,4,5,\dots\}$, we set $\zeta_q =\exp (2\pi \sqrt{-1} /q)$. 

By the definition of the reflection subgroup, we have
\begin{align*}
  \Mon^{(n)}(\al ,\be ,\ga) = \Ref^{(n)} (\al ,\be ,\ga) \cdot \langle M_1 ,\dots ,M_n \rangle .
\end{align*}
To determine the structure of $\Mon^{(n)}(\al ,\be ,\ga)$, we need to examine
the intersection $\Ref^{(n)} (\al ,\be ,\ga) \cap \langle M_1 ,\dots ,M_n \rangle$.
By the proof of Proposition \ref{prop:if-fin},
the reflection subgroup $\Ref^{(n)}$ can be decomposed into 
the product of some
$\Ref^{(1)}$'s or $\Ref^{(2)}$'s. 
% Their structures were studied in \cite{Kato} and \cite{Kato-Sekiguchi}. 
First, we consider the intersections
$\Ref^{(1)} \cap \langle M_1^{(1)} \rangle$ and 
$\Ref^{(2)} \cap \langle M_1^{(2)} , M_2^{(2)} \rangle$.
To examine these intersections, the following lemma is useful. 
\begin{Lem}\label{lem:str-order}
  % Let $Q \in \Mon^{(n)}$ be a matrix of order $q$.
  Suppose that $Q \in \Mon^{(n)}$ satisfies $Q^q =E_{2^n}$.
  If $(Q M_0^j)^{qr+1}=E_{2^n}$ with $r \in \Z$, then we have $Q \in \Ref^{(n)}$.
\end{Lem}
\begin{proof}
  By $Q^{-(q-1)} =Q$, we have
  \begin{align*}
    (Q M_0^j)^{q}
    =Q M_0^j Q^{-1} \cdot Q^2 M_0^j Q^{-2} \cdot Q^3 M_0^j Q^{-3}
    \cdots Q^{q-1} M_0^j Q^{-(q-1)} \cdot M_0^j \in \Ref^{(n)}. 
  \end{align*}
  Therefore, we obtain $Q = (Q M_0^j)^{-qr} \cdot M_0^{-j} \in \Ref^{(n)}$. 
\end{proof}

\subsection{A lemma on $\Mon^{(1)} (\al,\be ,\ga_1)$}
In this subsection, we set $c=c_1$ and $\ga=\ga_1$. 
% The intersection $\Ref^{(1)}(\al,\be,\ga) \cap \langle M_1 \rangle$ was studied in \cite[Section 6]{Kato}. 
% However, some cases were not mentioned. 
% For our convenience, we treat more cases. 
For our study, we need to examine the intersection $\Ref^{(1)}(\al,\be,\ga) \cap \langle M_1 \rangle$
when at least one of $\ga$ and $\be \al^{-1}$ is $-1$.
Although this intersection was studied in \cite{Kato}, all the cases were not considered.
To take all the cases into consideration, we improve the results in \cite[Section 6]{Kato}. 
\begin{Lem}\label{lem:str-Gauss}
  The intersection $\Ref^{(1)} \cap \langle M_1 \rangle$ is given as follows. 
  \begin{enumerate}[({I}-1)]
  \item\label{condition-I-1}
    If $\ga=\be \al^{-1}=-1$, then $\Ref^{(1)} \cap \langle M_1 \rangle =\{ E_2 \}$. 
  \item\label{condition-I-2}
    Assume $\ga=-1$ and $\be \al^{-1}\neq -1$.
    \begin{enumerate}[({I-\ref{condition-I-2}}-1)]
    \item\label{condition-I-2-1}
      If $\ga \al^{-1} \be^{-1}=-1$ and $\al$ is a primitive $q$th root of unity for an odd number $q$, then
      $\Ref^{(1)} \cap \langle M_1 \rangle =\langle M_1 \rangle$. 
    \item\label{condition-I-2-2}
      Otherwise, $\Ref^{(1)} \cap \langle M_1 \rangle =\{ E_2 \}$. 
    \end{enumerate}
  \item\label{condition-I-3}
    If $\ga \neq -1$ and $\be \al^{-1}=\ga \al^{-1} \be^{-1}=-1$,
    then $\Ref^{(1)} \cap \langle M_1 \rangle =\{ E_2 \}$.     
  \item\label{condition-I-4}
    Assume $\ga \neq -1$, $\be \al^{-1}=-1$ and $\ga \al^{-1} \be^{-1} \neq -1$.
    \begin{enumerate}[({I-\ref{condition-I-4}}-1)]
    \item\label{condition-I-4-1}
      If $\ga =\ga \al^{-1} \be^{-1}$ and it is a primitive 3rd root of unity, then 
      $\Ref^{(1)} \cap \langle M_1 \rangle =\langle M_1 \rangle$. 
    \item\label{condition-I-4-2}
      If both $\ga$ and $\ga \al^{-1} \be^{-1}$ are primitive 5th roots of unity, then 
      $\Ref^{(1)} \cap \langle M_1 \rangle =\langle M_1 \rangle$. 
    \item\label{condition-I-4-3}
      Otherwise, $\Ref^{(1)} \cap \langle M_1 \rangle =\{ E_2 \}$. 
    \end{enumerate}
  \end{enumerate}
\end{Lem}
\begin{proof}
  First, we assume $\ga =-1$.
  The claim (I-\ref{condition-I-1}) follows from \cite[Lemma 6.1]{Kato}.
  
  Second, we assume $\ga =-1$ and $\be \al^{-1}\neq -1$.
  % , and we prove the claim (I-\ref{condition-I-2}). 
  By \cite[Lemma 6.2]{Kato}, if $\ga \al^{-1} \be^{-1} \neq -1$, then we have
  $\Ref^{(1)} \cap \langle M_1 \rangle =\{ E_2 \}$.
  We consider the case when $\ga \al^{-1} \be^{-1} = -1$ and assume that $\al$ is
  a primitive $q$th root of unity.
  In this case, we have $\be=\al^{-1}$ and $M_0^2 =M_1^2 =E_2$.
  The characteristic polynomial of
  \begin{align*}
    M_1 M_0 =
    \begin{pmatrix}
      \al +\al^{-1}-1 & -1 \\ -\al -\al^{-1}+2 & 1
    \end{pmatrix}
  \end{align*}
  is $\phi (x)=x^2 -(\al +\al^{-1})x+1$, and its roots are $\al$ and $\al^{-1}$.
  Because of $\phi (x) | (x^q-1)$, we have $(M_1 M_0)^q =E_2$.
  It is sufficient to show that $M_1 \in \Ref^{(1)}$
  if and only if $q$ is an odd integer.
  If $q=2r+1$ ($r\in \Z_{>0}$), then
  $(M_1 M_0)^{2r+1} =E_2$ and Lemma \ref{lem:str-order} imply $M_1 \in \Ref^{(1)}$. 
  % we have
  % \begin{align*}
  %   M_1 =M_0 (M_1 M_0)^{2r}=M_0 (M_1 M_0 M_1 M_0)^{r}
  %   =M_0 (M_1 M_0 M_1^{-1} \cdot M_0)^{r} \in \Ref^{(1)} .
  % \end{align*}
  Conversely, we assume $M_1 \in \Ref^{(1)}$. 
  Since $\Ref^{(1)} =\langle M_0 , M_1 M_0 M_1\rangle$ and each generator is
  of order 2, $M_1$ is expressed as one of the following:
  \begin{align*}
    &(M_0 \cdot M_1 M_0 M_1)^r = M_0 (M_1 M_0)^{2r-1} M_1,\quad
      (M_0 \cdot M_1 M_0 M_1)^r M_0 = M_0 (M_1 M_0)^{2r},\\ 
    &(M_1 M_0 M_1 \cdot M_0)^r =(M_1 M_0)^{2r},\quad
      (M_1 M_0 M_1 \cdot M_0)^r M_1 M_0 M_1  =(M_1 M_0)^{2r+1} M_1 ,
  \end{align*}
  where $r$ is some integer.
  Since $\det M_1=-1$, we have $M_1 =M_0 (M_1 M_0)^{2r}$ or
  $M_1 =(M_1 M_0)^{2r+1} M_1$.
  In either case, we obtain $(M_1 M_0)^{2r+1}=E_2$,
  and hence, $\phi (x)$ divides $x^{2r+1}-1$.
  Since $\al$ is a root of $\phi (x)$, we have $\al^{2r+1}=1$. 
  Thus, $q$ is an odd number, and the claim (I-\ref{condition-I-2}) is proved. 

  Third, we assume $\ga \neq -1$ and $\be \al^{-1}=\ga \al^{-1} \be^{-1}=-1$.
  % and prove the claim (I-\ref{condition-I-3}). 
  Since $\det M_0=-1$ and $\det M_1 =\ga^{-1}$,
  the possibilities of 
  non-trivial intersection $\Ref^{(1)} \cap \langle M_1 \rangle$ are 
  only those cases in which $\ga$ is a primitive $2r$th root of unity ($r\in \Z_{>1}$)
  and $\Ref^{(1)} \cap \langle M_1 \rangle =\langle M_1^r \rangle$.
  We prove that these cases do not occur. 
  We assume $\ga$ is a primitive $2r$th root of unity. 
  By the assumption $\be \al^{-1}=\ga \al^{-1} \be^{-1}=-1$, we have
  \begin{align*}
    M_0=
    \begin{pmatrix}
      1 & 0 \\ \ga-1 & -1
    \end{pmatrix},\qquad 
    M_1=
    \begin{pmatrix}
      1 & -\ga^{-1} \\ 0 & \ga^{-1}
    \end{pmatrix}.
  \end{align*}
  By straightforward calculation,
  we can show that the characteristic polynomial of $M_1^j M_0$ is $x^2-\ga^{-j}$ and hence,
  we have $(M_1^j M_0)^2=\ga^{-j}E_2$. 
  We thus obtain 
  \begin{align*}
    M_1^j M_0 M_1^{-j} = M_1^j M_0 M_1^{2r-j}
    &= M_1^j \cdot M_1^j M_1^{2r-j} \cdot M_0 M_1^{2r-j} \cdot M_0 M_0 \\
    &= \ga^{-(2r-j)} M_1^{2j} M_0 = \ga^j M_1^{2j} M_0 .
  \end{align*}
  This expression and the identity $\ga^{r}=-1$ imply 
  $M_1^j M_0 M_1^{-j}=-M_1^{j-r} M_0 M_1^{-(j-r)}$.
  Especially for $j=r$, we have $M_1^r M_0 M_1^{-r}=-M_0$ and hence, we obtain 
  $-E_2 \in \Ref^{(1)}$. 
  Therefore, $\Ref^{(1)}$ is expressed as 
  \begin{align*}
    \Ref^{(1)}
    &=\langle M_0 , M_1 M_0 M_1^{-1}, M_1^2 M_0 M_1^{-2} ,\dots ,M_1^{2r-1} M_0 M_1^{-(2r-1)} \rangle \\
    &=\langle M_0 , \pm \ga M_1^2 M_0, \pm \ga^2 M_1^4 M_0 ,\dots  , \pm \ga^{r-1} M_1^{2(r-1)} M_0 ,-M_0\rangle \\
    &=\langle M_0 , -E_2 , \ga M_1^2 \rangle .
  \end{align*}
  Because of the identity 
  \begin{align*}
    (\ga M_1^2)^i M_0 (\ga M_1^2)^j
    &=\ga^{i+j} M_1^{2i}M_1^{2j} M_1^{-2j} M_0 M_1^{2j}
      =\ga^{i+j} M_1^{2i}M_1^{2j} \cdot \ga^{-2j} M_1^{-4j} M_0 \\
    &=\ga^{i-j} M_1^{2(i-j)} M_0
      =(\ga M_1^{2})^{i-j} M_0, 
  \end{align*}
  we obtain the expression 
  \begin{align*}
    \Ref^{(1)} =\{\pm (\ga M_1^2)^j \mid j=0,1,\dots ,r-1\}
    \cup \{\pm (\ga M_1^2)^j M_0 \mid j=0,1,\dots ,r-1\} .
  \end{align*}
  Comparing the entries of matrices, we can show 
  $M_1^r \not\in \Ref^{(1)}$.
  % only the case when $\ga =\pm \zeta_4$ and
  % $\Ref^{(1)} \cap \langle M_1 \rangle =\langle M_1^2 \rangle$.
  % % We assume $M_1^2 \in \Ref^{(1)}$.
  % We prove that this does not occur. 
  % We assume $\ga =\pm \zeta_4$. 
  % By straightforward calculation, we have
  % \begin{align*}
  %   &(M_0 M_1)^2=(M_1 M_0)^2 =-\ga E_2, \quad
  %     (M_1^2 M_0)^2=-E_2, \\
  %   & M_1 M_0 M_1^3 =\ga M_1^2 M_0 ,\quad
  %     M_1^3 M_0 M_1 =-\ga M_1^2 M_0 ,
  % \end{align*}
  % and hence we obtain 
  % \begin{align*}
  %   \Ref^{(1)}
  %   &=\langle M_0 , M_1 M_0 M_1^3, M_1^2 M_0 M_1^2 ,M_1^3 M_0 M_1 \rangle
  %     =\langle M_0 , \ga M_1^2 M_0, - M_0  , -\ga M_1^2 M_0 \rangle \\
  %   &=\langle M_0 , -E_2 , \ga M_1^2 M_0 \rangle 
  %     =\langle M_0 , -E_2 , \ga M_1^2 \rangle \\
  %   &=\{ \pm E_2 ,\ \pm (\ga M_1^2) ,\ \pm (\ga M_1^2)M_0, \ \pm (\ga M_1^2)M_0 (\ga M_1^2) \} .
  % \end{align*}
  % This expression shows $\ga E_2 \not\in \Ref^{(1)}$, which implies 
  % $M_1^2 \not\in \Ref^{(1)}$.
  Thus, the claim (I-\ref{condition-I-3}) is proved. 
  
  Finally, we assume $\ga \neq -1$, $\be \al^{-1}= -1$ and $\ga \al^{-1} \be^{-1} \neq -1$.
  % and prove the claim (I-\ref{condition-I-4}).
  When the triple $(\la,\mu,\nu)=(1-c,c-a-b,b-a)=(1-c,c-a-b,1/2)$ is one in
  the Schwarz list (\cite[Table 1]{Bod}, \cite{Schwarz}), our claim
  is already proved in \cite{Kato}.
  However, $\Mon^{(1)}$ is also finite irreducible when
  a triple $(\la,\mu,\nu)$ is one in the Schwarz list after performing
  the following operations: 
  permutation of $\la,\mu,\nu$;
  sign change of each of $\la,\mu,\nu$; 
  addition of $(l_1,l_2,l_3)\in \Z^3$ with $l_1+l_2+l_3 \in 2 \Z$. 
  A triple such as this was not considered in \cite{Kato}.
  Here, we consider all the cases. 
  As was mentioned in the proof of \cite[Lemma 6.3]{Kato},
  if the denominators of $c$ and $c-a-b$ are different, then
  we have $\Ref^{(1)} \cap \langle M_1 \rangle =\{ E_2 \}$.
  The remaining cases are
  \begin{align*}
    (\ga ,\ga \al^{-1} \be^{-1})
    &= (\zeta_3 ,\zeta_3) ,\ (\zeta_3 ,\zeta_3^2),\ (\zeta_3^2 ,\zeta_3) ,\ (\zeta_3^2 ,\zeta_3^2) , \\
    & (\zeta_5 ,\zeta_5^2) ,\ (\zeta_5 ,\zeta_5^3),\ (\zeta_5^4 ,\zeta_5^2) ,\ (\zeta_5^4 ,\zeta_5^3),\ 
      (\zeta_5^2 ,\zeta_5) ,\ (\zeta_5^3 ,\zeta_5),\ (\zeta_5^2 ,\zeta_5^4) ,\ (\zeta_5^3 ,\zeta_5^4) .
  \end{align*}
  In each case, the matrices $M_0$ and $M_1$ are defined over the field
  $\Q(\zeta_3)$ or $\Q(\zeta_5)$.
  The relations between $M_0$ and $M_1$ are preserved under the actions in 
  the Galois group ${\rm Gal}(\Q(\zeta_3)/\Q)\simeq \Z/ 2\Z$ or
  ${\rm Gal}(\Q(\zeta_5)/\Q)\simeq \Z/ 4\Z$. 
  Therefore, it suffices to consider the four cases:
  \begin{align*}
    (\ga ,\ga \al^{-1} \be^{-1})
    = (\zeta_3 ,\zeta_3) ,\ (\zeta_3^2 ,\zeta_3) ,\ (\zeta_5^4 ,\zeta_5^2) ,\ (\zeta_5^4 ,\zeta_5^3).
  \end{align*}
  Two cases $(\ga ,\ga \al^{-1} \be^{-1})=(\zeta_3^2 ,\zeta_3),\ (\zeta_5^4 ,\zeta_5^2)$
  are discussed in the proof of \cite[Lemma 6.3]{Kato}, and
  our claim is true.
  We discuss the other two cases.
  % If we assume $(\ga ,\ga \al^{-1} \be^{-1})=(\zeta_3 ,\zeta_3)$, then
  % we have $(M_1 M_0)^2 =(M_1 M_0^2)^3=-E_2$, and hence,
  % \begin{align*}
  %   M_1
  %   &=M_1 (M_1 M_0)^2 (M_1 M_0^2)^3
  %     =M_1^2 M_0 M_1 \cdot M_0 \cdot M_1 M_0^2 M_1^2 \cdot M_1^2 M_0^2 M_1 \cdot M_0^2 \\
  %   &=M_1^{-1} M_0 M_1 \cdot M_0 \cdot M_1 M_0^2 M_1^{-1} \cdot M_1^{-1} M_0^2 M_1 \cdot M_0^{-1} \in \Ref^{(1)}. 
  % \end{align*}
  % We assume $(\ga ,\ga \al^{-1} \be^{-1})=(\zeta_5^4 ,\zeta_5^3)$.
  % By strightforward calculation, we have $(M_1 M_0^2)^5 =(M_1 M_0^3)^3=-E_2$.
  % Because of
  % \begin{align*}
  %   -E_2
  %   &= (M_1 M_0^2)^5
  %     =M_1 M_0^2 M_1^{-1} \cdot M_1^2 M_0^2 M_1^{-2} \cdot M_1^3 M_0^2 M_1^{-3} \cdot
  %     M_1^{-1} M_0^2 M_1 \cdot M_0^2 
  %     \in \Ref^{(1)},  \\
  %   -M_1^2
  %   &= M_1^2 (M_1 M_0^3)^3
  %     =M_1^3 M_0^3 M_1^{-3} \cdot M_1^{-1} M_0^3 M_1 \cdot M_0^3 
  %     \in \Ref^{(1)} ,
  % \end{align*}
  % we obtain $M_1^2 \in \Ref^{(1)}$, and hence, $M_1 =(M_1^2)^3 \in \Ref^{(1)}$.
  If we assume $(\ga ,\ga \al^{-1} \be^{-1})=(\zeta_3 ,\zeta_3)$
  (resp. $(\zeta_5^4 ,\zeta_5^3)$), then
  we have $(M_1 M_0)^4=E_2$ 
  (resp. $(M_1 M_0^3)^6=E_2$) and hence,
  $M_1 \in \Ref^{(1)}$ by Lemma \ref{lem:str-order}. 
  Therefore, the proof of the claim (I-\ref{condition-I-4}) is completed. 
\end{proof}

\subsection{A lemma on $\Mon^{(2)} (\al,\be ,(\ga_1 ,\ga_2 ))$}
We consider the intersection $\Ref^{(2)} \cap \langle M_1 , M_2 \rangle$. 
If two of $\ga_1$, $\ga_2$, $\be \al^{-1}$ are $-1$,
then $\Ref^{(2)}$ is decomposed into $\Ref^{(1)}\times \Ref^{(1)}$
by Lemmas \ref{lem:red-1} and \ref{lem:red-2}.
These cases were already discussed in \cite[Section 7]{Kato}.
In this subsection, we assume that $\de_0^{(2)}=-1$, $\ga_1 \neq -1$ and 
either $\ga_2$ or $\be \al^{-1}$ is $-1$. 
Certain facts are verified by computer. 
\begin{Lem}\label{lem:str-F4}
  We assume $\de_0^{(2)}=-1$ (i.e., $\ga_1 \ga_2 \al^{-1} \be^{-1}=1$) and $\ga_1 \neq -1$. 
  The intersection $\Ref^{(2)} \cap \langle M_1 , M_2 \rangle$ is given as follows. 
  \begin{enumerate}[({II}-1)]
  \item\label{condition-II-1}
    If $\ga_2=-1$ and $\be \al^{-1} \neq -1$, then 
    $\Ref^{(2)} \cap \langle M_1 , M_2 \rangle =\{ E_4\}$. 
  \item\label{condition-II-2}
    Assume $\be \al^{-1} = -1$ and $\ga_2 \neq -1$.
    Let $\ga_k$ be a primitive $q_k$th root of unity ($q_k \in \{ 3,4,5\}$).
    \begin{enumerate}[({II-\ref{condition-II-2}}-1)]
    \item\label{condition-II-2-1}
      If $q_1 \neq q_2$, then
      $\Ref^{(2)} \cap \langle M_1 , M_2 \rangle =\{ E_4\}$. 
    \item\label{condition-II-2-2}
      If $q_1=q_2=3$ and $\ga_2=\ga_1$, then 
      $\Ref^{(2)} \cap \langle M_1 , M_2 \rangle=\{ E_4 \}$. 
    \item\label{condition-II-2-3}
      If $q_1=q_2=3$ and $\ga_2=\ga_1^2$, then 
      $\Ref^{(2)} \cap \langle M_1 , M_2 \rangle=\langle M_1 M_2 \rangle \simeq \Z/3\Z$. 
    \item\label{condition-II-2-4}
      If $q_1=q_2=5$, then 
      $\Ref^{(2)} \cap \langle M_1 , M_2 \rangle=\langle M_1 M_2^j \rangle \simeq \Z/5\Z$,
      where $j$ is an integer such that $\ga_1 \ga_2^{j}=1$. 
    \end{enumerate}
  \end{enumerate}
\end{Lem}
\begin{Rem}
  In the case (II-\ref{condition-II-2}),
  when $q_1=q_2$, we can show that $q_k \neq 4$.
  Indeed, if we assume $q_1=q_2=4$,
  the irreducibility condition (\ref{irred-2}) does not hold. 
\end{Rem}
\begin{proof}[Proof of Lemma \ref{lem:str-F4}]
  Since $\det M_0 =\de_0^{(2)}=-1$ and $\det (M_1^{j_1}M_2^{j_2})=\ga_1^{-2j_1}\ga_2^{-2j_2}$, 
  the intersection $\Ref^{(2)} \cap \langle M_1 , M_2 \rangle$ is trivial 
  except for the following possibilities:
  \begin{enumerate}[(i)]
  \item\label{proof-lem-F4-1}
    $\ga_2=-1$, $\be \al^{-1} \neq -1$; 
    % and $M_2 \in \Ref^{(2)} \cap \langle M_1 , M_2 \rangle$
    % (if $\ga_1 =\pm \zeta_4$, there are more possibilities);
  \item\label{proof-lem-F4-2}
    $\ga_2 \neq -1$, $\be \al^{-1} = -1$, $q_1 \neq q_2 =4$ and
    $M_2^2 \in \Ref^{(2)} \cap \langle M_1 , M_2 \rangle$; 
  \item\label{proof-lem-F4-3}
    $\ga_2 \neq -1$, $\be \al^{-1} = -1$, $q_1 = q_2 =3$, $\ga_2=\ga_1$ and
    $M_1 M_2^2 \in \Ref^{(2)} \cap \langle M_1 , M_2 \rangle$; 
  \item\label{proof-lem-F4-4}
    $\ga_2 \neq -1$, $\be \al^{-1} = -1$, $q_1 = q_2 =3$, $\ga_2=\ga_1^2$ and
    $M_1 M_2 \in \Ref^{(2)} \cap \langle M_1 , M_2 \rangle$; 
  \item\label{proof-lem-F4-5}
    $\ga_2 \neq -1$, $\be \al^{-1} = -1$, $q_1 = q_2 =5$, $\ga_1 \ga_2^{j}=1$ and
    $M_1 M_2^{j} \in \Ref^{(2)} \cap \langle M_1 , M_2 \rangle$. 
  \end{enumerate}
  % Note that in the case (\ref{proof-lem-F4-1}) (resp. (\ref{proof-lem-F4-2})--(\ref{proof-lem-F4-5})),
  % we have assumed $\be \al^{-1} \neq -1$ (resp. $\be \al^{-1} = -1$).  
  To prove the lemma, it suffices to show that
  (\ref{proof-lem-F4-1}), (\ref{proof-lem-F4-2}), (\ref{proof-lem-F4-3}) are false and 
  (\ref{proof-lem-F4-4}), (\ref{proof-lem-F4-5}) are true under our assumption. 
  We prove the fact that (\ref{proof-lem-F4-1}), (\ref{proof-lem-F4-2}), (\ref{proof-lem-F4-3})
  are false by using a computer\footnote{
    We use the system GAP (\texttt{https://www.gap-system.org}).
    The author does not have an elegant proof. 
    We use the functions
    \begin{itemize}
    \item \texttt{Group} and \texttt{NormalClosure} (\texttt{https://www.gap-system.org/Manuals/doc/ref/chap39.html})
      to define $\Mon^{(2)}$ and $\Ref^{(2)}$, respectively; 
    \item \texttt{Size} (\texttt{https://www.gap-system.org/Manuals/doc/ref/chap30.html})
      to compute the cardinalities of the groups $\Mon^{(2)}$ and $\Ref^{(2)}$. 
    \end{itemize}
  }.
  In the case (\ref{proof-lem-F4-1}), we have 18 possibilities of the pair $(\ga_1 ,\be \al^{-1})$
  up to the complex conjugate.
  Table \ref{tab:F4-check-1} lists the cardinalities of $\Mon^{(2)}$, $\Ref^{(2)}$ and
  the cyclic group $\langle M_1,M_2 \rangle$ for all the pairs.
  This implies $\Ref^{(2)} \cap \langle M_1 , M_2 \rangle =\{ E_4 \}$.
  Similarly, by computing the cardinalities (Table \ref{tab:F4-check-2}), we can verify that 
  (\ref{proof-lem-F4-2}) and (\ref{proof-lem-F4-3}) are also false. 
  \begin{table}
    \centering
    \begin{tabular}{|c|c|c|c|c|c|}
      \hline
      $\ga_1$ & $\be\al^{-1}$ & $|\Mon^{(2)}|$ & $|\Ref^{(2)}|$ & $|\Mon^{(2)}|/|\Ref^{(2)}|$ & $|\langle M_1,M_2 \rangle|$ \\
      \hline \hline
      $\zeta_3$ & $\zeta_3^j$ & 1152 & 192 & 6 & 6 \\ \hline 
      $\zeta_3$ & $\zeta_4^{2j-1}$ & 6912 & 1152 & 6 & 6 \\ \hline 
      $\zeta_3$ & $\zeta_5^k$ & 86400 & 14400 & 6 & 6 \\ \hline 
      $\zeta_4$ & $\zeta_3^j$ & 9216 & 1152 & 8 & 8 \\ \hline 
      $\zeta_5^i$ & $\zeta_3^j$ & 144000 & 14400 & 10 & 10 \\ \hline 
      $\zeta_5$ & $\zeta_5^2$ & 144000 & 14400 & 10 & 10 \\ \hline 
      $\zeta_5$ & $\zeta_5^3$ & 144000 & 14400 & 10 & 10 \\ \hline 
      $\zeta_5^2$ & $\zeta_5$ & 144000 & 14400 & 10 & 10 \\ \hline 
      $\zeta_5^2$ & $\zeta_5^4$ & 144000 & 14400 & 10 & 10 \\ \hline 
    \end{tabular}  
    \caption{(\ref{proof-lem-F4-1}) $\ga_2=\de_0^{(2)}=-1$ and $\be \al^{-1}\neq -1$ \ ($i,j\in \{ 1,2\}$, $k\in \{ 1,2,3,4\}$)}
    \label{tab:F4-check-1}
  \end{table}
  \begin{table}
    \centering
    \begin{tabular}{|c|c|c|c|c|c|}
      \hline
      $\ga_1$ & $\ga_2$ & $|\Mon^{(2)}|$ & $|\Ref^{(2)}|$ & $|\Mon^{(2)}|/|\Ref^{(2)}|$ & $|\langle M_1,M_2 \rangle|$ \\
      \hline \hline
      $\zeta_3$ & $\zeta_4$ & 13824 & 1152 & 12 & 12 \\ \hline 
      $\zeta_3^2$ & $\zeta_4$ & 13824 & 1152 & 12 & 12 \\ \hline
      $\zeta_3$ & $\zeta_3$ & 1728 & 192 & 9 & 9 \\ \hline
    \end{tabular}  
    \caption{$\be \al^{-1}=\de_0^{(2)}=-1$ and (\ref{proof-lem-F4-2}) $\ga_2=\zeta_4$,
    (\ref{proof-lem-F4-3}) $\ga_1=\ga_2=\zeta_3$}  
    \label{tab:F4-check-2}
  \end{table}

  We consider the case (\ref{proof-lem-F4-4}).
  We may assume $\ga_1 =\zeta_3$, $\ga_2 =\zeta_3^2$ and $\be \al^{-1} = -1$. 
  By straightforward calculation, we have $(M_1 M_2 M_0)^4=E_4$.
  Thus, Lemma \ref{lem:str-order} yields $M_1 M_2 \in \Ref^{(2)}$. 
  % Thus,
  % \begin{align*}
  %   M_1 M_2
  %   &=( M_0 M_1 M_2 M_0 M_1 M_2 M_0 M_1 M_2 M_0 )^{-1}\\
  %   &=( M_0 \cdot M_1 M_2 M_0 (M_1 M_2)^{-1} \cdot M_1^2 M_2^2 M_0 (M_1^2 M_2^2)^{-1} \cdot M_0 )^{-1}
  %     \in \Ref^{(2)} . 
  % \end{align*}

  Finally, we consider the case (\ref{proof-lem-F4-5}).
  By $\be \al^{-1} = -1$, $\ga_1 \al^{-1} \be^{-1}=\ga_2^{-1}$ and the Schwarz list,
  it is sufficient to discuss two cases $(\ga_1 ,\ga_2)=(\zeta_5,\zeta_5^2),(\zeta_5,\zeta_5^3)$,
  up to the actions in ${\rm Gal}(\Q(\zeta_5)/\Q)$.
  If we assume $(\ga_1 ,\ga_2)=(\zeta_5,\zeta_5^2)$, then we have
  $(M_1^3 M_2 M_0)^6=E_4$. 
  By Lemma \ref{lem:str-order}, we obtain $M_1^3 M_2 \in \Ref^{(2)}$ and hence, 
  $M_1 M_2^2 =(M_1^3 M_2)^2 \in \Ref^{(2)}$. 
  If we assume $(\ga_1 ,\ga_2)=(\zeta_5,\zeta_5^3)$, then we have
  $(M_1 M_2^3 M_0)^6=E_4$ which implies $M_1 M_2^3 \in \Ref^{(2)}$.
  Therefore, the proof is completed. 
\end{proof}

\subsection{Structure of $\Mon^{(n)}(\al,\be ,\ga)$}
Now, we provide the structure of $\Mon^{(n)}(\al,\be ,\ga)$.
Without loss of generality, 
we may assume that the condition (\ref{condition-B}-\ref{condition-B-a}), 
(\ref{condition-B}-\ref{condition-B-b}), (\ref{condition-B}-\ref{condition-B-c}) or
(\ref{condition-B}-\ref{condition-B-d}) holds.
Let $\ga_k$ be a primitive $q_k$th root of unity ($k\in \{1,2\}$, $q_k \in \{ 2,3,4,5,\dots\}$).
The structure of $\Mon^{(n)}(\al,\be ,\ga)$ is classified into the following four types: 
\begin{enumerate}[(Type 1)]
\item\label{type-1}
  $\Mon^{(n)}(\al ,\be ,\ga) = \Ref^{(n)} (\al ,\be ,\ga) \cdot \langle M_1 ,\dots ,M_n \rangle$ with
  \begin{align*}
    &\Ref^{(n)} \cap \langle M_1,\dots ,M_n \rangle =\{ E_{2^n} \}, \quad 
      \Ref^{(n)} (\al ,\be ,\ga) \simeq 
      \left(\Ref^{(1)}(\al,\be,\ga_1) \right)^{2^{n-1}},\\
    &\langle M_1 ,\dots ,M_n \rangle \simeq
      \Z/q_1 \Z \times (\Z/2\Z)^{n-1} ,\quad
      \Mon^{(1)}(\al,\be,\ga_1)/\Ref^{(1)}(\al,\be,\ga_1) \simeq \Z/q_1 \Z;
  \end{align*}
\item\label{type-2}
  $\Mon^{(n)}(\al ,\be ,\ga) = \Ref^{(n)} (\al ,\be ,\ga) \cdot \langle M_2 ,\dots ,M_n \rangle$ with
  \begin{align*}
    &\Ref^{(n)} \cap \langle M_2,\dots ,M_n \rangle =\{ E_{2^n} \}, \quad 
      \Ref^{(n)} (\al ,\be ,\ga) \simeq 
      \left(\Ref^{(1)}(\al,\be,\ga_1) \right)^{2^{n-1}},\\
    &\langle M_2 ,\dots ,M_n \rangle \simeq
      (\Z/2\Z)^{n-1} ,\quad
      \Mon^{(1)}(\al,\be,\ga_1) = \Ref^{(1)}(\al,\be,\ga_1);
  \end{align*}
\item\label{type-3}
  $\Mon^{(n)}(\al ,\be ,\ga) = \Ref^{(n)} (\al ,\be ,\ga) \cdot \langle M_1 ,\dots ,M_n \rangle$ with
  \begin{align*}
    &\Ref^{(n)} \cap \langle M_1,\dots ,M_n \rangle =\{ E_{2^n} \}, \quad 
      \Ref^{(n)}(\al,\be,\ga ) \simeq 
      \left(\Ref^{(2)}(\al,\be,(\ga_1,\ga_2)) \right)^{2^{n-2}}  ,\\
    &\langle M_1 ,\dots ,M_n \rangle \simeq
      \Z/q_1 \Z \times \Z/q_2 \Z \times (\Z/2\Z)^{n-2} ,\\
    &\Mon^{(2)}(\al,\be,(\ga_1,\ga_2))/\Ref^{(2)}(\al,\be,(\ga_1,\ga_2))
      \simeq \Z/q_1 \Z \times \Z/q_2 \Z;
  \end{align*}
\item\label{type-4} 
  $\Mon^{(n)}(\al ,\be ,\ga) = \Ref^{(n)} (\al ,\be ,\ga) \cdot \langle M_2 ,\dots ,M_n \rangle$ with
  \begin{align*}
    &\Ref^{(n)} \cap \langle M_2,\dots ,M_n \rangle =\{ E_{2^n} \}, \quad 
      \Ref^{(n)}(\al,\be,\ga ) \simeq 
      \left(\Ref^{(2)}(\al,\be,(\ga_1,\ga_2)) \right)^{2^{n-2}}  ,\\
    &\langle M_2 ,\dots ,M_n \rangle \simeq
      \Z/q_2 \Z \times (\Z/2\Z)^{n-2} ,\quad
    \Mon^{(2)}(\al,\be,(\ga_1,\ga_2))/\Ref^{(2)}(\al,\be,(\ga_1,\ga_2))
      \simeq \Z/q_2 \Z.
  \end{align*}
\end{enumerate}
Note that 
the structures of $\Ref^{(2)}(\al,\be,(\ga_1,\ga_2))$ in Type \ref{type-3} and Type \ref{type-4} 
were investigated in \cite{Kato-Sekiguchi}.

\begin{Th}\label{th:str}
  Let $n\geq 3$ and assume that $\Mon^{(n)}(\al,\be ,\ga)$ is finite irreducible. 
  We may also assume that $\ga_k$ is a primitive $q_k$th root of unity
  ($k\in \{1,2\}$, $q_k \in \{ 2,3,4,5,\dots\}$) and $\ga_3 =\cdots =\ga_n=-1$.  
  The structure of $\Mon^{(n)}(\al,\be ,\ga)$ is
  given as follows. 
  \begin{itemize}
  \item[(\ref{condition-B}-\ref{condition-B-a})]
    Assume $\ga_1=\ga_2=-1$ ($q_1=q_2=2$).
    \begin{enumerate}[({\ref{condition-B}-\ref{condition-B-a}}-1)]
    \item\label{list-str-Ba1}
      If $\be \al^{-1}\neq -1$, $\al \be=1$ and
      $\al$ is a primitive $q$th root of unity for an odd number $q$, then
      $\Mon^{(n)}(\al,\be ,\ga)$ is of Type 2 and
      $\Ref^{(n)} \cap \langle M_1 ,\dots ,M_n \rangle =\langle M_1 \cdots M_n \rangle$. 
    \item\label{list-str-Ba2}
      Otherwise, $\Mon^{(n)}(\al,\be ,\ga)$ is of Type 1. 
    \end{enumerate}
  \item[(\ref{condition-B}-\ref{condition-B-b})]
    Assume $\ga_2=\be\al^{-1}=-1$ and $\ga_1 \neq -1$. 
    \begin{enumerate}[({\ref{condition-B}-\ref{condition-B-b}}-1)]
    \item\label{list-str-Bb1}
      If $\ga_1 \al^{-1} \be^{-1} \neq -1$, $\ga_1 =\ga_1 \al^{-1} \be^{-1}$ and $q_1=3$, then
      $\Mon^{(n)}(\al,\be ,\ga)$ is of Type 2 and
      $\Ref^{(n)} \cap \langle M_1 ,\dots ,M_n \rangle =\langle M_1  \rangle$.       
    \item\label{list-str-Bb2}      
      If $\ga_1 \al^{-1} \be^{-1} \neq -1$ and both $\ga_1$ and $\ga_1 \al^{-1} \be^{-1}$ are
      primitive 5th roots of unity, then
      $\Mon^{(n)}(\al,\be ,\ga)$ is of Type 2 and
      $\Ref^{(n)} \cap \langle M_1 ,\dots ,M_n \rangle =\langle M_1  \rangle$.       
    \item\label{list-str-Bb3}      
      Otherwise, $\Mon^{(n)}(\al,\be ,\ga)$ is of Type 1. 
    \end{enumerate}
  \item[(\ref{condition-B}-\ref{condition-B-c})] 
    Assume $\ga_2=\de_0^{(n)}(\al,\be,\ga)=-1$ and $\ga_1 \neq -1$, $\be\al^{-1} \neq -1$.
    Then $\Mon^{(n)}(\al,\be ,\ga)$ is of Type 3. 
  \item[(\ref{condition-B}-\ref{condition-B-d})] 
    Assume $\be\al^{-1}=\de_0^{(n)}(\al,\be,\ga)=-1$ and $\ga_1 \neq -1$, $\ga_2 \neq -1$.
    \begin{enumerate}[({\ref{condition-B}-\ref{condition-B-d}}-1)]
    \item\label{list-str-Bd1}
      If $q_1 =q_2=3$ and $\ga_2 =\ga_1^2$, then
      $\Mon^{(n)}(\al,\be ,\ga)$ is of Type 4 and
      $\Ref^{(n)} \cap \langle M_1 ,\dots ,M_n \rangle =\langle M_1 M_2 \rangle$.       
    \item\label{list-str-Bd2}      
      If $q_1=q_2=5$, then
      $\Mon^{(n)}(\al,\be ,\ga)$ is of Type 4 and
      $\Ref^{(n)} \cap \langle M_1 ,\dots ,M_n \rangle =\langle M_1 M_2^j \rangle$,
      where $j$ is an integer such that $\ga_1 \ga_2^{j}=1$. 
    \item\label{list-str-Bd3}      
      Otherwise, $\Mon^{(n)}(\al,\be ,\ga)$ is of Type 3. 
    \end{enumerate}
  \end{itemize}
\end{Th}
\begin{proof}
  We use an approach similar to that of \cite[Theorems 7.1 and 7.2]{Kato} 
  to prove the theorem. 
  We only have to discuss the intersection
  $\Ref^{(n)} \cap \langle M_1 ,\dots ,M_n \rangle$.

  First, we consider the case (\ref{condition-B}-\ref{condition-B-a}).
  In the proof of Lemma \ref{lem:red-1}, 
  $W^{\pm}$ are invariant under $\Ref^{(n)}$, while $M_{n-1}$, $M_{n}$ interchange
  $W^{+}$ and $W^{-}$.
  Because $M_1,\dots ,M_{n-2}$ preserve $W^{+}$ and $W^{-}$, we have 
  \begin{align}
    \label{eq:M-prod-not-in-1}
    M_1^{i_1} \cdots M_{n-2}^{i_{n-2}} M_{n-1}\not\in \Ref^{(n)}, 
    \qquad (i_1,\dots ,i_{n-2})\in \Z^{n-2}.
  \end{align}
  By rearranging the indices $\{ 1,\dots ,n\}$, we obtain
  $M_1^{i_1}\cdots M_n^{i_n} \not\in \Ref^{(n)}$, except for
  $E_{2^n}$ and $M_1 \cdots M_n$. 
  Therefore, we obtain 
  $\Ref^{(n)} \cap \langle M_1 ,\dots ,M_n \rangle =\Ref^{(n)} \cap \langle M_1 \cdots M_n \rangle$.
  Repeated applications of the decomposition in the proof of Lemma \ref{lem:red-1} show that
  $\C^{2^n}$ is decomposed into a direct sum of two-dimensional subspaces.
  The restriction of the action of $M_1 \cdots M_n$ and $\Ref^{(n)}$ on
  each of these two-dimensional subspaces are $M_1^{(1)} \in \Mon^{(1)}(\al,\be ,-1)$
  and $\Ref^{(1)}(\al,\be ,-1)$, respectively. 
  The intersection $\Ref^{(n)} \cap \langle M_1 ,\dots ,M_n \rangle$ coincides with
  $\langle M_1 \cdots M_n \rangle$ if and only if
  $\Ref^{(1)}(\al,\be ,-1) \cap \langle M_1^{(1)} \rangle =\langle M_1^{(1)} \rangle$. 
  Thus, our claim follows from 
  Lemma \ref{lem:str-Gauss} (I-\ref{condition-I-1}) and (I-\ref{condition-I-2}). 

  Next, we consider the case (\ref{condition-B}-\ref{condition-B-b}).
  In the proof of Lemma \ref{lem:red-2}, 
  $W^{\pm}$ are invariant under $\Ref^{(n)}$, whereas $M_{n}$ interchange
  $W^{+}$ and $W^{-}$.
  Because $M_1,\dots ,M_{n-1}$ preserve $W^{+}$ and $W^{-}$, we have 
  \begin{align}
    \label{eq:M-prod-not-in-2}
    M_1^{i_1} \cdots M_{n-1}^{i_{n-1}} M_{n}\not\in \Ref^{(n)},
    \qquad (i_1,\dots ,i_{n-1})\in \Z^{n-1}.
  \end{align}
  By rearranging the indices $\{ 2,\dots ,n\}$, we obtain
  $M_1^{i_1}\cdots M_n^{i_n} \not\in \Ref^{(n)}$, except for
  $M_1^{i_1}$. 
  Therefore, we obtain 
  $\Ref^{(n)} \cap \langle M_1 ,\dots ,M_n \rangle =\Ref^{(n)} \cap \langle M_1 \rangle$.
  Repeated applications of the decomposition in the proof of Lemma \ref{lem:red-2} show that
  $\C^{2^n}$ is decomposed into a direct sum of two-dimensional subspaces.
  The restriction of the action of $M_1$ and $\Ref^{(n)}$ on
  each of these two-dimensional subspaces are $M_1^{(1)} \in \Mon^{(1)}(\al,\be ,\ga_1)$
  and $\Ref^{(1)}(\al,\be ,\ga_1)$, respectively. 
  The intersection $\Ref^{(n)} \cap \langle M_1 ,\dots ,M_n \rangle$ coincides with
  $\langle M_1 \rangle$ if and only if
  $\Ref^{(1)}(\al,\be ,\ga_1) \cap \langle M_1^{(1)} \rangle =\langle M_1^{(1)} \rangle$. 
  Thus, our claim follows from 
  Lemma \ref{lem:str-Gauss} (I-\ref{condition-I-3}) and (I-\ref{condition-I-4}). 

  Finally, we consider the cases 
  (\ref{condition-B}-\ref{condition-B-c}) and (\ref{condition-B}-\ref{condition-B-d}).
  % Similarly to the above discussion, 
  % we can show the following claims. 
  \begin{itemize}
  \item[(\ref{condition-B}-\ref{condition-B-c})]
    Similarly to the case (\ref{condition-B}-\ref{condition-B-a}), we can show 
    $\Ref^{(n)} \cap \langle M_1 ,\dots ,M_n \rangle
    =\Ref^{(n)} \cap \langle M_1 , M_2 \cdots M_n \rangle$. 
    Note that in this case, 
    we can rearrange the indices $\{ 2,\dots ,n\}$ in (\ref{eq:M-prod-not-in-1}). 
    % Repeated applications of the decomposition in the proof of Lemma \ref{lem:red-1} show that
    % $\C^{2^n}$ is decomposed into a direct sum of four-dimensional subspaces.
    % The restriction of the action of $M_1$, $M_2 \cdots M_n$ and $\Ref^{(n)}$ on 
    We also have the decomposition of $\C^{2^n}$ into a direct sum of four-dimensional subspaces, 
    and the restriction of the action of $M_1$, $M_2 \cdots M_n$ and $\Ref^{(n)}$ on 
    each of these four-dimensional subspaces are $M_1^{(2)},M_2^{(2)} \in \Mon^{(2)}(\al,\be ,(\ga_1,\ga_2))$
    and $\Ref^{(2)}(\al,\be ,(\ga_1,\ga_2))$, respectively. 
  \item[(\ref{condition-B}-\ref{condition-B-d})] 
    Similarly to the case (\ref{condition-B}-\ref{condition-B-b}), we can show 
    $\Ref^{(n)} \cap \langle M_1 ,\dots ,M_n \rangle
    =\Ref^{(n)} \cap \langle M_1 , M_2 \rangle$. 
    Note that in this case, 
    we can rearrange the indices $\{ 3,\dots ,n\}$ in (\ref{eq:M-prod-not-in-2}). 
    % Repeated applications of the decomposition in the proof of Lemma \ref{lem:red-2} show that
    % $\C^{2^n}$ is decomposed into a direct sum of four-dimensional subspaces.
    % The restriction of the action of $M_1$, $M_2$ and $\Ref^{(n)}$ on 
    We also have the decomposition of $\C^{2^n}$ into a direct sum of four-dimensional subspaces, 
    and the restriction of the action of $M_1$, $M_2$ and $\Ref^{(n)}$ on 
    each of these four-dimensional subspaces are $M_1^{(2)},M_2^{(2)} \in \Mon^{(2)}(\al,\be ,(\ga_1,\ga_2))$
    and $\Ref^{(2)}(\al,\be ,(\ga_1,\ga_2))$, respectively. 
  \end{itemize}
  Therefore, the structure of the intersection $\Ref^{(n)} \cap \langle M_1 ,\dots ,M_n \rangle$
  is determined according to that of $\Ref^{(2)} \cap \langle M_1^{(2)} , M_2^{(2)} \rangle$.
  By Lemma \ref{lem:str-F4}, our claims are proved. 
\end{proof}


\begin{thebibliography}{99}
% \bibitem{AK}
% K. Aomoto and M. Kita, 
% ``Theory of Hypergeometric Functions'', 
% translated by K. Iohara, 
% Springer Monographs in Mathematics,
% Springer-Verlag, Tokyo, 2011. 
%
% \bibitem{Beukers}
% F. Beukers,
% \emph{Monodromy of A-hypergeometric functions}, preprint.
%
\bibitem{BH}
  F. Beukers and G. Heckman,
  \emph{Monodromy for the hypergeometric function ${}_nF_{n-1}$}, 
  {Invent. math.} \textbf{95} (1989), Issue 2, 325--354.
%
\bibitem{Beukers} 
  F. Beukers,
  Algebraic $A$-hypergeometric functions, 
  {\it Invent. math.} {\bf 180} (2010), 589--610. 
%
\bibitem{Bod} 
  E. Bod,
  Algebraicity of the Appell--Lauricella and Horn hypergeometric functions,
  {\it J. Differ. Equations} {\bf 252} (2012), 541--566. 
%
\bibitem{G-FC-monodromy}
  Y. Goto,
  \emph{The monodromy representation of Lauricella's hypergeometric function $F_C$}, 
  {Ann. Sc. Norm. Super. Pisa Cl. Sci. (5)}, \textbf{XVI} (2016), 1409--1445. 
  % 
\bibitem{GK-FC-pi1}
  Y. Goto and J. Kaneko, 
  \emph{The fundamental group of the complement of 
    the singular locus of Lauricella's $F_C$}, 
  {J. Singul.}, \textbf{17} (2018), 295--329.
  % 
\bibitem{GK-FC-Zariski}
  Y. Goto and K. Koike, 
  \emph{Picard--Vessiot groups of Lauricella's hypergeometric systems $E_C$ and 
    Calabi-Yau varieties arising integral representations},
  to appear in \emph{ J. Lond. Math. Soc}.
  % preprint, arXiv:1807.10890.
  % 
\bibitem{GM-FC} Y. Goto and K. Matsumoto, 
  Irreducibility of the monodromy representation of Lauricella's $F_C$,
  \emph{Hokkaido Math. J.} \textbf{48} (2019), 489--512. 
  % 
\bibitem{HT}
  R. Hattori and N. Takayama, 
  \emph{The singular locus of Lauricella's $F_C$}, 
  {J. Math. Soc. Japan}, \textbf{66} (2014), 981--995.
%
% \bibitem{Kaneko}
% J. Kaneko, 
% \emph{Monodromy group of Appell's system $(F_{4})$}, 
% {Tokyo J. Math.}, \textbf{4} (1981), 35--54. 
%
% \bibitem{Kato}
% M. Kato, 
% \emph{The Riemann problem for Appell's $F_4$}, 
% Mem. Fac. Sci. Kyushu Univ. Ser. A \textbf{47} (1993), 227--243.
%
\bibitem{Kato} M. Kato, 
  Appell's $F_4$ with finite irreducible monodromy group, 
  {\it Kyushu J. Math.} {\bf 51} (1997), no. 1, 125--147. 
  %
\bibitem{Kato-Sekiguchi} M. Kato and J. Sekiguchi, 
  Reflection subgroups of the monodromy groups of Appell's $F_4$, 
  {\it Kyushu J. Math.} {\bf 64} (2010), no. 2, 281--296. 
  % 
\bibitem{L}
  G. Lauricella, 
  {\it Sulle funzioni ipergeometriche a pi\`u variabili}, 
  {Rend. Circ. Mat. Palermo}, 
  \textbf{7} (1893), 111--158.
%
\bibitem{Sasaki} T. Sasaki,
  Picard--Vessiot group of Appell's system of hypergeometric differential equations 
  and infiniteness of monodromy group, 
  {\it Kumamoto J. Sci.(Math.)} {\bf 14} (1980/81), no. 1, 85--100. 
  % 
\bibitem{Schwarz} H. A. Schwarz, 
  Ueber diejenigen F\"alle, 
  in welchen die Gaussische hypergeometrische Reihe eine algebraische Function 
  ihres vierten Elementes darstellt. 
  {\it J. Reine Angew. Math.} {\bf 75} (1873) 292--335. 
%
\bibitem{Terasoma}
  T. Terasoma, 
  {\it Fundamental group of non-singular locus of Lauricella's $F_C$}, 
  preprint, arXiv:1803.06609. 
\end{thebibliography}
\end{document}